\documentclass[10pt]{amsart}
\usepackage{amsmath}
\usepackage{amssymb}
\usepackage{amsfonts}
\usepackage{amsthm}
\usepackage{enumerate}
\usepackage{hyperref}
\usepackage{color}
\usepackage{psfrag}
\usepackage[all]{xy}
\usepackage{amscd}
\usepackage{verbatim}
\usepackage{calc}

\numberwithin{equation}{section}
\theoremstyle{plain}
\newtheorem{thm}{Theorem}[section]
\newtheorem{pro}[thm]{Proposition}
\newtheorem{lem}[thm]{Lemma}
\newtheorem{cor}[thm]{Corollary}

\theoremstyle{definition}
\newtheorem{defn}[thm]{Definition}

\newtheorem{cons}[thm]{Construction}

\newtheorem{rem}[thm]{Remark}
\newtheorem{f}[thm]{ }

\begin{document}

\title[Presentations of rings with a chain of semidualizing modules]
{Presentations of rings with a chain of semidualizing modules}


     \author[E. Amanzadeh]{Ensiyeh Amanzadeh}
     \author[M. T. Dibaei]{Mohammad T. Dibaei}


\address{Ensiyeh Amanzadeh, School of Mathematics,
Institute for Research in Fundamental Sciences (IPM),  P.O. Box:
19395-5746,  Tehran,  Iran; and Faculty of Mathematical Sciences and Computer,
Kharazmi University,  Tehran,  Iran. }

\email{en.amanzadeh@gmail.com}

\address{Mohammad T. Dibaei, Faculty of Mathematical Sciences and Computer,
Kharazmi University,  Tehran,  Iran; and School of Mathematics,
Institute for Research in Fundamental Sciences (IPM),  P.O. Box:
19395-5746,  Tehran,  Iran. }

\email{dibaeimt@ipm.ir}








\begin{abstract}
Inspired by Jorgensen et.\! al., it is proved that if a  Cohen-Macaulay local ring $R$ with dualizing module admits a suitable chain of semidualizing $R$-modules of length $n$, then $R\cong Q/(I_1+\cdots+I_n)$ for some Gorenstein ring $Q$ and ideals $I_1,\cdots, I_n$ of $Q$; and, for each $\Lambda\subseteq [n]$, the ring $Q/(\Sigma_{l\in \Lambda} I_l)$ has some interesting cohomological properties . This extends the result of Jorgensen et.\!~al., and also of Foxby and  Reiten.
\end{abstract}


\maketitle

\section{Introduction} \label{sec1}

Throughout  $R$ is a commutative noetherian local ring.
Foxby \cite{f},  Vasconcelos \cite{v} and Golod \cite{g} independently initiated the study of
semidualizing modules.
A finite (i.e. finitely generated) $R$-module $C$ is called \emph{semidualizing} if the natural
homothety map  $\chi_C^R: R \longrightarrow \mathrm{Hom}_R(C,  C)$ is an isomorphism
and $\mathrm{Ext}^{\geqslant1}_R(C,  C)=0$ ( see \cite[Definition 1.1]{hj} ).
Examples of semidualizing $R$-modules include $R$ itself and a dualizing $R$-module when one exists.
The set of all isomorphism classes of semidualizing $R$-modules is denoted by $\mathfrak{G}_0(R)$,
and the isomorphism class of a semidualizing $R$-module $C$ is denoted $[C]$.
The set $\mathfrak{G}_0(R)$ have caught attentions of several authors; see, for example \cite{fs-w2}, \cite{cs-w}, \cite{ns-w} and \cite{s-w1}.
In \cite{cs-w}, Christensen and Sather-Wagstaff  show that $\mathfrak{G}_0(R)$ is finite when $R$ is Cohen-Macaulay and equicharacteristic.
Then Nasseh and Sather-Wagstaff, in \cite{ns-w}, settle the general assertion that
$\mathfrak{G}_0(R)$ is finite.
Also, in \cite{s-w1}, Sather-Wagstaff studies  the cardinality of $\mathfrak{G}_0(R)$.

Each semidualizing $R$-module $C$ gives rise to a notion of reflexivity for finite $R$-modules.
For instance, each finite projective $R$-module is  totally $C$-reflexive.
For semidualizing $R$-modules $C$ and $B$, we write $[C] \trianglelefteq [B]$ whenever
$B$ is totally $C$-reflexive.
In \cite{gerko}, Gerko defines chains in $\mathfrak{G}_0(R)$. A \emph{chain} in
$\mathfrak{G}_0(R)$ is a sequence
$[C_n]\trianglelefteq \cdots \trianglelefteq[C_1] \trianglelefteq [C_0]$,
and such a chain has length $n$ if $[C_i]\neq[C_j]$ whenever $i\neq j$.
In \cite{s-w1}, Sather-Wagstaff uses the length of chains in $\mathfrak{G}_0(R)$
to provide a lower bound for the cardinality of $\mathfrak{G}_0(R)$.

It is well-known that a Cohen-Macaulay ring which is homomorphic
image of a Gorenstein local ring, admits a dualizing module (see \cite[Theorem 3.9]{s} ).
Then Foxby \cite{f} and Reiten \cite{reiten}, independently, prove the converse.
Recently Jorgensen et.\! al. \cite{jls-w}, characterize the Cohen-Macaulay local rings which admit
dualizing modules and non-trivial semidualizing modules ( i.e. neither free nor dualizing ).

In this paper, we are interested in characterization of a  Cohen-Macaulay ring $R$ which admits
a dualizing module and a certain chain in $\mathfrak{G}_0(R)$.
We prove that, when a Cohen-macaulay ring $R$ with dualizing module has a
\emph{suitable chain} in $\mathfrak{G}_0(R)$ (see Definition~\ref{suitable}) of length $n$,
then there exist a Gorenstein ring $Q$ and ideals $I_1, \cdots, I_n$ of $Q$ such that
$R\cong Q/(I_1+\cdots+I_n)$ and, for each $\Lambda\subseteq [n]=\{1, \cdots, n\}$,
the ring $Q/(\Sigma_{l\in \Lambda} I_l)$
has certain  homological and cohomological properties (see Theorem~\ref{T42}).
Note that, this result gives the result of Jorgensen et.\! al. when $n=2$ and the result of
Foxby and Reiten in the case $n=1$.
We prove a partial converse of Theorem~\ref{T42} in Propositions \ref{P42} and \ref{P43}.

\section{Preliminaries} \label{sec2}

This section contains definitions and background material.
\begin{defn}\label{reflexive} (\cite[Definition 2.7]{hj} and \cite[Theorem 5.2.3 and Definition 6.1.2]{s-w2})
Let $C$ be a semidualizing $R$-module. A finite $R$-module $M$ is \emph{totally} $C$-\emph{reflexive}
when it satisfies the following conditions\,:
\begin{itemize}
\item[(i)] the natural homomorphism
$\delta^C_M: M \longrightarrow \mathrm{Hom}_R(\mathrm{Hom}_R(M, C), C)$
is an isomorphism, and
\item[(ii)]
$\mathrm{Ext}^{\geqslant1}_R(M,  C)=0=\mathrm{Ext}^{\geqslant1}_R(\mathrm{Hom}_R(M, C),  C)$.
\end{itemize}
A totally $R$-reflexive is referred to as totally reflexive.
The $\mbox{G}_C$-dimension of a finite $R$-module $M$, denoted $\mbox{G}_C$-$\mathrm{dim}_R(M)$,
is defined as

$\hspace{1cm}\mbox{G}_C$-$\mathrm{dim}_R(M)= \mathrm{inf}\left\{ \begin{array}{lll}
n\geqslant 0 & {\Bigg |} & {\small\begin{array}{l}
\mathrm{there\ is\ an\ exact\ sequence\ of\ } R\!-\!\mathrm{modules}\\ 0\rightarrow G_n\rightarrow \cdots\rightarrow G_1\rightarrow G_0\rightarrow M\rightarrow0\\ \mathrm{such\ that\ each\ } G_i\ \mathrm{is\ totally}\  C\!-\!\mathrm{reflexive}\end{array}}
\end{array}\right\}.$
\end{defn}

\begin{rem}\label{R1}  \cite[Theorem 6.1]{c2}
Let $S$ be a Cohen-Macaulay local ring equipped with a module-finite local ring
homomorphism $\tau: R \rightarrow S$ such that $R$ is Cohen-Macaulay.
Assume that $C$ is a semidualizing $R$-module.
Then $\mbox{G}_C$-$\mathrm{dim}_R(S)<\infty$ if and
only if there exists an integer $g\geqslant0$ such that $\mathrm{Ext}^i_R(S, C) = 0$ for all $i$,
$i\neq g$, and $\mathrm{Ext}^g_R(S, C)$
is a semidualizing $S$-module; when these conditions hold, one has $g=\mbox{G}_C$-$\mathrm{dim}_R(S)$.
\end{rem}

\begin{f}\label{f2}
Define the order $\trianglelefteq$ on $\mathfrak{G}_0(R)$. For $[B], [C]\in \mathfrak{G}_0(R)$,
write  $[C]\trianglelefteq[B]$ when $B$ is totally $C$-reflexive (see, e.g., \cite{s-w1}).
This relation is reflexive and antisymmetric \cite[Lemma 3.2]{fs-w1},
but it is not known whether it is transitive in general.
Also, write $[C]\vartriangleleft[B]$ when $[C]\trianglelefteq [B]$ and $[C]\neq[B]$.
For a semidualizing $C$, set
$$\mathfrak{G}_C(R)= \big\{[B]\in \mathfrak{G}_0(R)\  \big|\  [C]\trianglelefteq [B] \big\}.$$
In the case  $D$ is a dualizing $R$-module, one has $[D]\trianglelefteq [B]$ for any semidualizing $R$-module $B$, by \cite[(V.2.1)]{h}, and so $\mathfrak{G}_D(R)= \mathfrak{G}_0(R)$.

If $[C]\trianglelefteq[B]$ then $\mathrm{Hom}_R(B, C)$ is a semidualizing and
$[C]\trianglelefteq[\mathrm{Hom}_R(B, C)]$ (\cite[Theorem 2.11]{c2}).
Moreover, if $A$ is another semidualizing $R$-module with $[C]\trianglelefteq[A]$, then $[B]\trianglelefteq[A]$
if and only if $[\mathrm{Hom}_R(A, C)]\trianglelefteq[\mathrm{Hom}_R(B, C)]$ (\cite[Proposition 3.9]{fs-w1}).
\end{f}

\begin{thm}\label{T01}\cite[Theorem 3.1]{gerko}
Let $B$ and $C$ be two semidualizing $R$-modules such that $[C]\trianglelefteq[B]$.
Assume that $M$ is an $R$-module which is both totally $B$-reflexive and totally $C$-reflexive,
then the composition map
$$\varphi : \mathrm{Hom}_R(M, B)\otimes_R\mathrm{Hom}_R(B, C) \longrightarrow
\mathrm{Hom}_R(M, C)$$
is an isomorphism.
\end{thm}
\begin{cor}\label{C01}\cite[Corollary 3.3]{gerko}
If $[C_n]\trianglelefteq \cdots \trianglelefteq[C_1] \trianglelefteq [C_0]$ is a chain in $\mathfrak{G}_0(R)$, then one gets
$$C_n\cong C_0 \otimes_R \mathrm{Hom}_R(C_0,  C_1) \otimes_R \cdots \otimes_R\mathrm{Hom}_R(C_{n-1},  C_n).$$
\end{cor}

Assume that  $[C_n]\vartriangleleft \cdots \vartriangleleft[C_1] \vartriangleleft [C_0]$
is a chain in $\mathfrak{G}_0(R)$.
For each $i\in [n]$ set $B_i=\mathrm{Hom}_R(C_{i-1}, C_i)$.
For each sequence of integers $\mathbf{i} = \{i_1, \cdots , i_j\}$  with $j \geqslant 1$
and $1\leqslant i_1 <  \cdots  < i_j \leqslant n$, set
$B_{\mathbf{i}}=B_{i_1}\otimes_R\cdots \otimes_R B_{i_j}$.
( $B_{\{i_1\}} =B_{i_1}$ and set $B_\emptyset=C_0$.)

In order to facilitate the discussion, we list some results from \cite{s-w1}. We first recall the following definition.
\begin{defn}
Let $C$ be a semidualizing $R$-module.
The \emph{Auslander class} $\mathcal{A} _C(R)$ with respect to $C$
is the class of all $R$-modules $M$ satisfying the following conditions.

(1) The natural map $ \gamma^C_M: M \longrightarrow \mathrm{Hom}_R(C,  C\otimes_R M)$ is an isomorphism.

(2) $\mathrm{Tor}^R_{\geqslant1}(C,  M)=0=\mathrm{Ext}^{\geqslant1}_R(C,  C\otimes_R M)$.
\end{defn}
\begin{pro}\label{P01}
Assume that  $[C_n]\vartriangleleft \cdots \vartriangleleft[C_1] \vartriangleleft [C_0]$
is a chain in $\mathfrak{G}_0(R)$ such that \\
$\mathfrak{G}_{C_1}(R)\subseteq\mathfrak{G}_{C_2}(R)\subseteq\cdots \subseteq \mathfrak{G}_{C_n}(R)$.
\begin{itemize}
\item[(1)] \cite[Lemma 4.3]{s-w1} For each $i, p$, $1\leqslant i\leqslant i+p\leqslant n$
$$B_{\{i, i+1,\cdots, i+p\}}\cong \mathrm{Hom}_R(C_{i-1}, C_{i+p}).$$
\item[(2)] \cite[Lemma 4.4]{s-w1} If $1\leqslant i< j-1 \leqslant n-1$, then
$$B_{\{i, j\}}\cong \mathrm{Hom}_R(\mathrm{Hom}_R(B_i, C_{j-1}), C_j).$$
\item[(3)] \cite[Lemma 4.5]{s-w1} For each sequence $\mathbf{i}=\{i_1, \cdots , i_j\}\subseteq [n]$,
the $R$-module $B_{\mathbf{i}}$ is a semidualizing.
\item[(4)] \cite[Lemma 4.6]{s-w1} If $\mathbf{i}=\{i_1, \cdots , i_j\}\subseteq [n]$ and $\mathbf{s}=\{s_1, \cdots , s_t\}\subseteq [n]$
are two sequences with $\mathbf{s}\subseteq \mathbf{i}$, then
$[B_{\mathbf{i}}] \trianglelefteq[B_{\mathbf{s}}]$ and
$\mathrm{Hom}_R(B_{\mathbf{s}}, B_{\mathbf{i}})\cong B_{\mathbf{i}\setminus \mathbf{s}}$.
\item[(5)] \cite[Theorem 4.11]{s-w1} If $\mathbf{i}=\{i_1, \cdots , i_j\}\subseteq [n]$ and $\mathbf{s}=\{s_1, \cdots , s_t\}\subseteq [n]$
are two sequences, then the following conditions are equivalent.
\begin{itemize}
\item[(i)] $B_{\mathbf{i}} \in \mathcal{A}_{B_{\mathbf{s}}}(R)$.
\item[(ii)] $B_{\mathbf{s}} \in \mathcal{A}_{B_{\mathbf{i}}}(R)$.
\item[(iii)] The $R$-module $B_{\mathbf{i}}\otimes_R B_{\mathbf{s}} $ is semidualizing.
\item[(iv)] $ \mathbf{i}\cap \mathbf{s}=\emptyset $.
\end{itemize}
\end{itemize}
\end{pro}

At the end of this section we recall the definition of trivial extension ring.
Note that this notion is the main key in
the proof of the converse of Sharp's result \cite{s}, which is given by Foxby \cite{f} and Reiten \cite{reiten}.
\begin{f}\label{f5}
For an $R$-module $M$, the \emph{trivial extension} of $R$ by $M$ is the ring $R \ltimes M$,
described as follows. As an $R$-module, we have $R \ltimes M=R\oplus M$. The multiplication
is defined by $(r, m)(r', m')=(rr', rm'+r'm)$.
Note that the composition $R\rightarrow R \ltimes M\rightarrow R$ of the natural homomorphisms is the identity map of $R$.

Note that, for a semidualizing $R$-module $C$,
the trivial extension ring $R\ltimes C$ is a commutative noetherian local ring.
If $R$ is Cohen-Macaulay then $R\ltimes C$ is Cohen-Macaulay too.
For more information about the trivial extension rings one may see, e.g., \cite[Section 2]{jls-w}.
\end{f}

\section{Results} \label{sec3}

This section is devoted to the main result, Theorem~\ref{T42}, which extends the results
of Jorgensen et.\!~al. \cite[Theorem 3.2]{jls-w} and of Foxby \cite{f} and  Reiten \cite{reiten}.

For a semidualizing $R$-module $C$, set ${(-)}^{\dag_C}=\mathrm{Hom}_R(-,  C)$.
The following notations are taken from \cite{s-w1}.

\begin{defn}\label{suitable}
Let  $[C_n]\vartriangleleft \cdots \vartriangleleft[C_1] \vartriangleleft [C_0]$ be a chain in $\mathfrak{G}_0(R)$ of length $n$.
For each sequence of integers $\mathbf{i} = \{i_1, \cdots , i_j\}$ such that $j \geqslant 0$ and $1\leqslant i_1 <  \cdots  < i_j \leqslant n$, set
$C_{\mathbf{i}}= C_0^{\dag_{C_{i_1}}\dag_{C_{i_2}}\cdots\dag_{C_{i_j}}}$.
(When $j = 0$, set  $C_{\mathbf{i}} = C_\emptyset= C_0$ ).

We say that the above chain is \emph{suitable} if  $C_0=R$ and
$C_{\mathbf{i}}$ is totally $C_t$-reflexive, for all
$\mathbf{i}$ and $t$ with $i_j\leqslant t \leqslant n$.

Note that if $[C_n]\vartriangleleft \cdots \vartriangleleft[C_1] \vartriangleleft [R]$
is a suitable chain, then $C_{\mathbf{i}}$ is a semidualizing $R$-module for each
$\mathbf{i}\subseteq [n]$.
Also, for each sequence of integers $\{ x_1, \cdots, x_m\}$ with
$1\leqslant x_1<\cdots <x_m\leqslant n$,  the sequence
$[C_{x_m}]\vartriangleleft \cdots \vartriangleleft[C_{x_1}] \vartriangleleft [R]$
is a suitable chain in $\mathfrak{G}_0(R)$ of length $m$.
\end{defn}

Sather-Wagstaff, in \cite[Theorem 3.3]{s-w1}, proves that if $\mathfrak{G}_0(R)$ admits a chain
$[C_n]\vartriangleleft \cdots \vartriangleleft[C_1] \vartriangleleft [C_0]$ such that
$\mathfrak{G}_{C_0}(R)\subseteq\mathfrak{G}_{C_1}(R)\subseteq\cdots \subseteq \mathfrak{G}_{C_n}(R)$, then $|\mathfrak{G}_0(R)|\geqslant 2^n$.
Indeed, the classes $[C_{\mathbf{i}}]$, which are parameterized by the
allowable sequences $\mathbf{i}$, are precisely the $2^n$
classes constructed in the proof of \cite[Theorem 3.3]{s-w1}.

\begin{thm}\cite[Theorem 4.7]{s-w1}\label{T03}
Let $\mathfrak{G}_0(R)$ admit  a chain
$[C_n]\vartriangleleft \cdots \vartriangleleft[C_1] \vartriangleleft [C_0]$
such that
$\mathfrak{G}_{C_1}(R)\subseteq\mathfrak{G}_{C_2}(R)\subseteq\cdots \subseteq \mathfrak{G}_{C_n}(R)$.  If $C_0=R$,
then the $R$-modules $B_{\mathbf{i}}$ are precisely
the $2^n$ semidualizing modules constructed in \cite[Theorem 3.3]{s-w1}.
\end{thm}

\begin{rem}\label{R40}
In Proposition~\ref{P01} and Theorem~\ref{T03},
if we replace the assumption of existence of a chain
$[C_n]\vartriangleleft \cdots \vartriangleleft[C_1] \vartriangleleft [C_0]$
in $\mathfrak{G}_0(R)$ such that
$\mathfrak{G}_{C_1}(R)\subseteq\mathfrak{G}_{C_2}(R)\subseteq\cdots \subseteq \mathfrak{G}_{C_n}(R)$
by the existence of a suitable chain, then the assertions hold true as well.
\end{rem}


The next lemma and  proposition give us sufficient tools to treat Theorem~\ref{T42}.

\begin{lem}\label{L41}
Assume that $R$ admits a suitable chain
$[C_n]\vartriangleleft \cdots \vartriangleleft[C_1] \vartriangleleft [C_0]=[R]$  in $\mathfrak{G}_0(R)$.
Then for any $k\in [n]$, there exists a suitable chain
\begin{equation}\label{e40}
[C_n]\vartriangleleft \cdots \vartriangleleft[C_{k+1}]\vartriangleleft[C_k]\vartriangleleft[C_1^{\dag_{C_k}}]\vartriangleleft \cdots\vartriangleleft[C_{k-2}^{\dag_{C_k}}] \vartriangleleft[C_{k-1}^{\dag_{C_k}}] \vartriangleleft [R]
\end{equation}
in $\mathfrak{G}_0(R)$ of length $n$.
\end{lem}
\begin{proof}
For $i,j$, $0\leqslant j<i\leqslant k$, as $[C_i]\vartriangleleft[C_j]$
one has
$[C_{j}^{\dag_{C_k}}]\vartriangleleft[C_{i}^{\dag_{C_k}}]$.
As $[C_k]\neq[C_{i}^{\dag_{C_k}}]$, one gets
$[C_t]\vartriangleleft[C_{i}^{\dag_{C_k}}]$ for each $t$, $k\leqslant t\leqslant n$.
Thus  (\ref{e40}) is a chain in $\mathfrak{G}_0(R)$ of length $n$.

Next, we show that (\ref{e40}) is a suitable chain.
For $r, t \in\{0, 1, \cdots , n\}$ and a sequence $\{x_1,\cdots, x_m\}$ of integers with
$r\leqslant x_1<\cdots < x_m\leqslant t$, repeated use of Theorem~\ref{T01} implies
$$C_r^{\dag_{C_t}}\cong C_r^{\dag_{C_{x_1}}}\otimes_R C_{x_1}^{\dag_{C_{x_2}}}\otimes_R \cdots \otimes_R C_{x_m}^{\dag_{C_t}}.$$

For each $r$, $0<r<k$, set $C'_r=C_r^{\dag_{C_k}}$.
If $\mathbf{i} = \{i_1, \cdots , i_j\}$ and $\mathbf{u} = \{u_1, \cdots , u_s\}$
are sequences of integers such that $j, s\geqslant 0$ and
$1\leqslant i_j <  \cdots  < i_1 < k\leqslant u_1< \cdots  < u_s\leqslant n$, then we set
$$C_{\mathbf{i},\mathbf{u}}=
C_0^{\dag_{C'_{i_1}}\cdots\dag_{C'_{i_j}}\dag_{C_{u_1}}\cdots\dag_{C_{u_s}}}.$$
When $s=0$ (resp., $j=0$ or $j=0=s$) we have $C_{\mathbf{i}, \mathbf{u}}=C_{\mathbf{i}, \emptyset}$
(resp.,  $C_{\mathbf{i}, \mathbf{u}}=C_{\emptyset, \mathbf{u}}$ or
$C_{\mathbf{i}, \mathbf{u}}=C_{\emptyset, \emptyset} =C_0$).

By Proposition~\ref{P01}(4) and Remark~\ref{R40}, one has
$C_0^{\dag_{C'_{i_1}}\dag_{C'_{i_2}}}\cong\mathrm{Hom}_R(C_{i_1}^{\dag_{C_k}}, C_{i_2}^{\dag_{C_k}})\cong C_{i_2}^{\dag_{C_{i_1}}}$ and so
$C_0^{\dag_{C'_{i_1}}\dag_{C'_{i_2}}\dag_{C'_{i_3}}}\cong
 \mathrm{Hom}_R(C_{i_2}^{\dag_{C_{i_1}}}, C_{i_3}^{\dag_{C_k}})\cong C_{i_3}^{\dag_{C_{i_2}}}\otimes_R C_{i_1}^{\dag_{C_k}}$.
By proceeding in this way one obtains the following isomorphism
\begin{equation}\label{e41}
C_0^{\dag_{C'_{i_1}}\cdots\dag_{C'_{i_j}}}\cong\left\{ \begin{array}{ll}
C_{i_j}^{\dag_{C_{i_{j-1}}}}\otimes_R C_{i_{j-2}}^{\dag_{C_{i_{j-3}}}}\otimes_R\cdots\otimes_R C_{i_2}^{\dag_{C_{i_1}}} & \mathrm{if}\ j\ \mathrm{is\ even}, \\
  &   \\
C_{i_j}^{\dag_{C_{i_{j-1}}}}\otimes_R C_{i_{j-2}}^{\dag_{C_{i_{j-3}}}}\otimes_R\cdots\otimes_R C_{i_1}^{\dag_{C_k}} &  \mathrm{if}\ j\ \mathrm{is\ odd}.
\end{array} \right.
\end{equation}
Therefore, by Proposition~\ref{P01}(2) and Remark~\ref{R40},
$$C_0^{\dag_{C'_{i_1}}\cdots\dag_{C'_{i_j}}}\cong\left\{ \begin{array}{ll}
C_0^{\dag_{C_{i_j}}\cdots\dag_{C_{i_1}}} & \mathrm{if}\ j\ \mathrm{is\ even}, \\
  &   \\
C_0^{\dag_{C_{i_j}}\cdots\dag_{C_{i_1}}\dag_{C_k}} &  \mathrm{if}\ j\ \mathrm{is\ odd},
\end{array} \right.$$
and thus
$$C_{\mathbf{i}, \mathbf{u}}\cong\left\{ \begin{array}{ll}
C_0^{\dag_{C_{i_j}}\cdots\dag_{C_{i_1}}\dag_{C_{u_1}}\cdots\dag_{C_{u_s}}} & \mathrm{if}\ j\ \mathrm{is\ even}, \\
  &   \\
C_0^{\dag_{C_{i_j}}\cdots\dag_{C_{i_1}}\dag_{C_k}\dag_{C_{u_1}}\cdots\dag_{C_{u_s}}} &  \mathrm{if}\ j\ \mathrm{is\ odd}.
\end{array} \right.$$
Hence, by assumption, $[C_t]\trianglelefteq[C_{\mathbf{i}, \mathbf{u}}]$ for all $t$, $t\geqslant u_s$.
If $s=0$, then
$C_{\mathbf{i}, \mathbf{u}}=C_{\mathbf{i}, \emptyset}=C_0^{\dag_{C'_{i_1}}\cdots\dag_{C'_{i_j}}}$. \\
On the other hand, for each $l$, $1\leqslant l\leqslant i_j$,  we have
$$C_l^{\dag_{C_k}}\cong C_l^{\dag_{C_{i_j}}}\otimes_R C_{i_j}^{\dag_{C_{i_{j-1}}}}\otimes_R \cdots
\otimes_R C_{i_3}^{\dag_{C_{i_2}}}\otimes_R C_{i_2}^{\dag_{C_{i_1}}}\otimes_R C_{i_1}^{\dag_{C_k}}.$$
Thus, by Proposition~\ref{P01}(4) and (\ref{e41}),
$[C_l^{\dag_{C_k}}]\trianglelefteq [C_{\mathbf{i}, \mathbf{u}}]$.
Hence the chain (\ref{e40}) is suitable.
\end{proof}

\begin{rem}\label{R42}
Let $R$ be Cohen-Macaulay and  $[C_n]\vartriangleleft \cdots \vartriangleleft[C_1] \vartriangleleft [C_0]$ be a suitable chain in $\mathfrak{G}_0(R)$.
For any $k$,  $1\leqslant k\leqslant n$, set $R_k=R\ltimes C_{k-1}^{\dag_{C_k}}$ the trivial extension of $R$ by $C_{k-1}^{\dag_{C_k}}$.
Then $R_k$ is totally $C_l^{\dag_{C_k}}$-reflexive and totally $C_t$-reflexive $R$-module for all $l, t$ with $1\leqslant l< k \leqslant t\leqslant n$.
Set
$$C_l^{(k)}=\left\{ \begin{array}{ll}
\mathrm{Hom}_R(R_k, C_{k-1-l}^{\dag_{C_k}}) &  \text{if}\ \ 0\leqslant l< k-1\\
\mathrm{Hom}_R(R_k, C_{l+1}) &   \text{if}\ \ k-1\leqslant l\leqslant n-1 \hspace{0.1cm}.
\end{array} \right.$$
Then, by Remark~\ref{R1}, $C_l^{(k)}$ is a semidualizing $R_k$-module for all $l$, $0\leqslant l\leqslant n-1$.
\end{rem}

\begin{pro}\label{P41}
Under the hypotheses of Remark \ref{R42}, for all $k$, $1\leqslant k\leqslant n$,
$$[C_{n-1}^{(k)}] \vartriangleleft  \cdots  \vartriangleleft [C_1^{(k)}] \vartriangleleft  [R_k]$$
is a suitable chain in $\mathfrak{G}_0(R_k)$ of length $n-1$.
\end{pro}
\begin{proof}
Let $k\in[n]$.
For integers $a, b$ with $a\neq b$ and  $0\leqslant a, b\leqslant n-1$,
we observe that $[C_a^{(k)}]\neq[C_b^{(k)}]$. Indeed, we consider the three cases
$0\leqslant a, b < k-1$, $0\leqslant a<k-1 \leqslant b \leqslant n-1$, and
$k-1 \leqslant a, b \leqslant n-1$. We only discuss about the first case. The other cases are treated in a similar way.  For $0\leqslant a, b < k-1$,
if $[C_a^{(k)}]=[C_b^{(k)}]$, then
$\hspace{0.2cm}\mathrm{Hom}_R(R_k, C_{k-1-a}^{\dag_{C_k}})\cong \mathrm{Hom}_R(R_k, C_{k-1-b}^{\dag_{C_k}})\hspace{0.1cm}$ and so
$\mathrm{Hom}_{R_k}(R, \mathrm{Hom}_R(R_k, C_{k-1-a}^{\dag_{C_k}}))\cong \mathrm{Hom}_{R_k}(R, \mathrm{Hom}_R(R_k, C_{k-1-b}^{\dag_{C_k}}) )$.
Thus, by adjointness, $C_{k-1-a}^{\dag_{C_k}} \cong C_{k-1-b}^{\dag_{C_k}}$ ,
which contradicts with (\ref{e40}) in Lemma~\ref{L41}.

In order to proceed with the proof, for an $R_k$-module $M$, we invent the symbol
$(-)^{\dag^k_M}=\mathrm{Hom}_{R_k}(-,  M)$.
Note that, for $R_k$-modules $M_1, \cdots, M_t$, we have
$$(-)^{\dag^k_{M_1} \dag^k_{M_2} \cdots \dag^k_{M_t}}=
\bigg(\Big(\big( (-)^{\dag^k_{M_1}}\big)^{\dag^k_{M_2}}\Big)^{\cdots}\bigg)^{\dag^k_{M_t}}=
\mathrm{Hom}_{R_k}\big((-)^{\dag^k_{M_1} \dag^k_{M_2} \cdots \dag^k_{M_{t-1}}}, M_t\big).$$

For two sequences of integers $\mathbf{p} = \{p_1, \cdots , p_r\}$ and $\mathbf{q} = \{q_1, \cdots , q_s\}$ such that $r, s\geqslant 0$ and
$0< p_1 <  \cdots  < p_r < k-1\leqslant q_1< \cdots  < q_s\leqslant n-1$, set
$$C^{(k)}_{\mathbf{p}, \mathbf{q}}=
R_k^{\dag^k_{C_{p_1}^{(k)}}\cdots\dag^k_{C_{p_r}^{(k)}}\dag^k_{C_{q_1}^{(k)}}\cdots
\dag^k_{C_{q_s}^{(k)}}}.$$
Therefore one gets the following $R$-module isomorphisms
$$\begin{array}{llll}
C^{(k)}_{\mathbf{p}, \mathbf{q}}& = &
\mathrm{Hom}_{R_k}(\cdots \mathrm{Hom}_{R_k}(\mathrm{Hom}_{R_k}(
\cdots \mathrm{Hom}_{R_k}(R_k, C_{p_1}^{(k)})\cdots , C_{p_r}^{(k)}),
C_{q_1}^{(k)}) \cdots , C_{q_s}^{(k)})\\
 & \cong  &
\mathrm{Hom}_{R}(\cdots \mathrm{Hom}_{R}(\mathrm{Hom}_{R}(
\cdots \mathrm{Hom}_{R}(R_k, C_{k-1-p_1}^{\dag_{C_k}}) \cdots , C_{k-1-p_r}^{\dag_{C_k}}), C_{q_1+1}) \cdots , C_{q_s+1})\\
& \cong &
R^{\dag_{C'_{k-1-p_1}}\cdots\dag_{C'_{k-1-p_r}}\dag_{C_{q_1+1}}\cdots\dag_{C_{q_s+1}}}
\oplus
R^{\dag_{C'_{k-1}}\dag_{C'_{k-1-p_1}}\cdots\dag_{C'_{k-1-p_r}}\dag_{C_{q_1+1}}\cdots\dag_{C_{q_s+1}}}\\
& = & C_{\mathbf{i}, \mathbf{u}}\oplus C_{\mathbf{i'}, \mathbf{u}}\, ,
\end{array}$$
where  $\mathbf{i} = \{k-1-p_1, \cdots , k-1-p_r\}$, $\mathbf{i'} = \{k-1, k-1-p_1, \cdots , k-1-p_r\}$,
 $\mathbf{u} = \{q_1+1, \cdots , q_s+1\}$, $C'_l=C_l^{\dag_{C_k}}$, for all $0< l< k$,
 and $C_{\mathbf{i}, \mathbf{u}}$ and  $C_{\mathbf{i'}, \mathbf{u}}$ are as in the
proof of Lemma~\ref{L41}.

As $[C_{t+1}]\trianglelefteq [C_{\mathbf{i}, \mathbf{u}}]$ and
$[C_{t+1}]\trianglelefteq [C_{\mathbf{i'}, \mathbf{u}}]$ in $\mathfrak{G}_0(R)$
for all $t$, $q_s\leqslant t\leqslant n-1$, one gets
$[C^{(k)}_t]\trianglelefteq [C^{(k)}_{\mathbf{p}, \mathbf{q}}]$ in
$\mathfrak{G}_0(R_k)$, by \cite[Theorem 6.5]{c2}.
When $s=0$ we have
$C^{(k)}_{\mathbf{p}, \mathbf{q}}=C^{(k)}_{\mathbf{p}, \emptyset}\cong C_{\mathbf{i}, \emptyset}\oplus C_{\mathbf{i'}, \emptyset}$.
By Lemma~\ref{L41}, for all $m$, $p_r\leqslant m<k-1$, one has
$[C_{k-1-m}^{\dag_{C_k}}]\trianglelefteq [C_{\mathbf{i}, \emptyset}]$ and
$[C_{k-1-m}^{\dag_{C_k}}]\trianglelefteq [C_{\mathbf{i'}, \emptyset}]$ in $\mathfrak{G}_0(R)$.
Thus, by \cite[Theorem 6.5]{c2}, one gets
$[C^{(k)}_m]\trianglelefteq [C^{(k)}_{\mathbf{p}, \emptyset}]$ in $\mathfrak{G}_0(R_k)$.
Hence $[C_{n-1}^{(k)}] \vartriangleleft  \cdots  \vartriangleleft [C_1^{(k)}] \vartriangleleft  [R_k]$
is a suitable chain in $\mathfrak{G}_0(R_k)$ of length $n-1$.
\end{proof}

To state our main result, we recall the definitions of Tate homology and Tate cohomology ( see \cite{am} and \cite{jls-w} for more details ).
\begin{defn}\label{Tate1}
Let $M$ be a finite $R$-module. A \emph{Tate resolution} of $M$ is a diagram
$\mathbf{T}\stackrel{\vartheta}{\longrightarrow}\mathbf{P}\stackrel{\pi}{\longrightarrow}M$, where $\pi$ is an $R$-projective resolution of $M$,
$\mathbf{T}$ is an exact complex of projectives such that $\mathrm{Hom}_R(T, R)$ is exact, $\vartheta$ is a morphism, and $\vartheta_i$ is isomorphism for all $i\gg 0$.

By \cite[Theorem 3.1]{am}, a finite $R$-module has finite $\mbox{G}$-dimension if and only if it admits a Tate resolution.
\end{defn}

\begin{defn}\label{Tate2}
Let $M$ be a finite $R$-module of finite $\mbox{G}$-dimension, and let
$\mathbf{T}\stackrel{\vartheta}{\longrightarrow}\mathbf{P}\stackrel{\pi}{\longrightarrow}M$ be a Tate resolution of $M$.
For each integer $i$ and each $R$-module $N$, the $i$th \emph{Tate homology} and \emph{Tate cohomology} modules are
$$\widehat{\mathrm{Tor}}^R_i(M, N)= \mathrm{H}_i(\mathbf{T}\otimes_R N) \hspace{1.7cm} \widehat{\mathrm{Ext}}_R^i(M, N)=\mathrm{H}_{-i}(\mathrm{Hom}_R(\mathbf{T}, N)).$$
\end{defn}

\begin{thm}\label{T42}
Let $R$ be a Cohen-Macaulay ring with a dualizing module $D$.
Assume that $R$ admits a suitable chain
$[C_n]\vartriangleleft \cdots \vartriangleleft[C_1] \vartriangleleft [R]$  in $\mathfrak{G}_0(R)$
and  that $C_n\cong D$. Then there exist a Gorenstein local ring $Q$ and
ideals $I_1, \cdots, I_n$ of $Q$, which satisfy the following conditions.
In this situation, for each $\Lambda\subseteq [n]$, set $R_{_\Lambda}=Q/(\Sigma_{l\in \Lambda} I_l)$, in particular $R_{_\emptyset}=Q$.
\begin{itemize}
\item[(1)] There is a ring isomorphism $R\cong Q/{(I_1+\cdots+I_n)}$.
\item[(2)]  For each $\Lambda\subseteq [n]$ with $\Lambda\neq \emptyset$, the ring $R_{_\Lambda}$ is non-Gorenstein Cohen-Macaulay with a dualizing module.
\item[(3)]  For each $\Lambda\subseteq [n]$ with $\Lambda\neq \emptyset$,  we have $\bigcap_{l\in \Lambda}I_l=\prod_{l\in \Lambda}I_l$.
\item[(4)]  For subsets $\Lambda$, $\Gamma$ of $[n]$ with $\Gamma\subsetneq\Lambda$, we have $\mathrm{G}$-$\dim_{R_{_\Gamma}} R_{_{\Lambda}}=0$, and
$\mathrm{Hom}_{R_{_{\Gamma}}} (R_{_{\Lambda}}, R_{_{\Gamma}})$ is a non-free semidualizing $R_{_{\Lambda}}$-module.
\item[(5)] For subsets $\Lambda$, $\Gamma$ of $[n]$ with $\Lambda\neq \Gamma$, the module
$\mathrm{Hom}_{R_{_{\Lambda\cap \Gamma}}} (R_{_\Lambda}, R_{_{\Gamma}})$
is not cyclic  and
$$\mathrm{Ext}^{\geqslant1}_{R_{_{\Lambda\cap \Gamma}}} (R_{_\Lambda}, R_{_{\Gamma}})=0=\mathrm{Tor}^{R_{_{\Lambda\cap \Gamma}}}_{\geqslant1}(R_{_\Lambda}, R_{_{\Gamma}}).$$
\item[(6)] For subsets $\Lambda$, $\Gamma$ of $[n]$ with $|\Lambda\setminus \Gamma|=1$, we have
$$\mathrm{\widehat{Ext}}^i_{R_{_{\Lambda\cap \Gamma}}} (R_{_\Lambda}, R_{_{\Gamma}})=0=\mathrm{\widehat{Tor}}^{R_{_{\Lambda\cap \Gamma}}}_i(R_{_\Lambda}, R_{_{\Gamma}})$$
for all $i \in \mathbb{Z}$.
\end{itemize}
\end{thm}

The ring $Q$ is constructed as an iterated trivial extension of $R$. As an $R$-module,
it has the form $Q=\oplus_{\mathbf{i}\subseteq [n]} B_{\mathbf{i}}$.
The details are contained in the following construction.
\begin{cons}\label{cons}
We construct the ring $Q$ by induction on $n$.
We claim that the ring $Q$, as an $R$-module, has the form
$Q=\oplus_{\mathbf{i}\subseteq [n]} B_{\mathbf{i}}$ and the ring structure on it is as follows.\\
For two elements $\big(\alpha_{\mathbf{i}}\big)_{\mathbf{i}\subseteq [n]}$ and $\big(\theta_{\mathbf{i}}\big)_{\mathbf{i}\subseteq [n]}$  of $Q$
$$\big(\alpha_{\mathbf{i}}\big)_{\mathbf{i}\subseteq [n]} \big(\theta_{\mathbf{i}}\big)_{\mathbf{i}\subseteq [n]}=\big(\sigma_{\mathbf{i}}\big)_{\mathbf{i}\subseteq [n]}\  , \  \text{where}\ \ \ \ \sigma_{\mathbf{i}}=\sum_{\scriptsize{\begin{array}{c}\mathbf{v}\subseteq \mathbf{i}\\
\mathbf{w}= \mathbf{i}\setminus \mathbf{v}  \end{array}} }
 \alpha_{\mathbf{v}}\cdot\theta_{\mathbf{w}}\, .$$

For $n=1$, set $Q=R\ltimes C_1$ and $I_1=0\oplus C_1$, which is the result of Foxby \cite{f} and
Reiten \cite{reiten}.
The case $n=2$ is proved by Jorgensen et.\!~al. \cite[Theorem 3.2]{jls-w}.
They proved that the extension ring $Q$ has the form
$Q=R\oplus C_1 \oplus C_1^{\dag_{C_2}}\oplus C_2$ as an $R$-module
(i.e.  $Q=B_\emptyset\oplus B_1\oplus B_2 \oplus B_{\{1,2\}}$).
Also the ring structure on $Q$ is  given by
$(r, c, f, d)(r', c', f', d')=(rr', rc' + r'c, rf'+r'f, f'(c)+f(c')+rd'+r'd)$.
The ideal $I_l$, $l=1,2$, has the form $I_l=0\oplus 0\oplus B_l \oplus B_{\{1,2\}}$.

Let  $n>2$.  Take an element $k\in [n]$. By Proposition~\ref{P41}, the ring
$R_k=R\ltimes C_{k-1}^{\dag_{C_k}}$
has the suitable chain
$[C_{n-1}^{(k)}] \vartriangleleft  \cdots  \vartriangleleft [C_1^{(k)}] \vartriangleleft  [R_k]$
in $\mathfrak{G}_0(R_k)$ of length $n-1$.
Note that $C_{n-1}^{(k)}=\mathrm{Hom}_R(R_k, C_n)\cong\mathrm{Hom}_R(R_k, D)$
is a dualizing $R_k$-module.

We  set $B_i^{(k)}=\mathrm{Hom}_{R_k}(C_{i-1}^{(k)}, C_{i}^{(k)})$, $i=1, \cdots, n-1$.
For two sequences $\mathbf{p}=\{p_1, \cdots, p_r\}$, $\mathbf{q}=\{q_1, \cdots, q_s\}$
such that $r, s\geqslant 1$ and
$1\leqslant p_1 <  \cdots  < p_r < k-1\leqslant q_1< \cdots  < q_s\leqslant n-1$, we set
\begin{equation}\label{e46}
B^{(k)}_{\mathbf{p}, \mathbf{q}}=B_{p_1}^{(k)}\otimes_{R_k}\cdots\otimes_{R_k}B_{p_r}^{(k)}\otimes_{R_k} B_{q_1}^{(k)}\otimes_{R_k}\cdots\otimes_{R_k}B_{q_s}^{(k)},
\end{equation}
$$B^{(k)}_{\mathbf{p}, \emptyset}=B_{p_1}^{(k)}\otimes_{R_k}\cdots\otimes_{R_k}B_{p_r}^{(k)},\ \
B^{(k)}_{\emptyset, \mathbf{q}}= B_{q_1}^{(k)}\otimes_{R_k}\cdots\otimes_{R_k}B_{q_s}^{(k)},\ \ \mathrm{and}\ \
B^{(k)}_{\emptyset, \emptyset}=C_0^{(k)}=R_k.$$
By applying the induction hypothesis on $R_k$ there is an extension ring, say $Q_k$, which is
Gorenstein local and, as an $R_k$-module, has the form
$$Q_k= \bigoplus_{\scriptsize{\begin{array}{c}\mathbf{p}\subseteq \{1, \cdots, k-2\}\\
\mathbf{q}\subseteq \{k-1, \cdots, n-1\}\end{array}}} B^{(k)}_{\mathbf{p}, \mathbf{q}}\ .$$
Moreover, the ring structure on $Q_k$ is as follows.\\
For
$\phi=\big(\phi_{\mathbf{p}, \mathbf{q}}\big)_{\tiny{
\mathbf{p}\subseteq \{1, \cdots, k-2\},\
\mathbf{q}\subseteq \{k-1, \cdots, n-1\}}}$
and
$\varphi=\big(\varphi_{\mathbf{p}, \mathbf{q}}\big)_{\tiny{
\mathbf{p}\subseteq \{1, \cdots, k-2\},\
\mathbf{q}\subseteq \{k-1, \cdots, n-1\}}}$
of $Q_k$

\begin{equation}\label{e45}
\phi\, \varphi=\psi=\big(\psi_{\mathbf{p}, \mathbf{q}}\big)_{\tiny{
\mathbf{p}\subseteq \{1, \cdots, k-2\},\
\mathbf{q}\subseteq \{k-1, \cdots, n-1\}}}\  , \ \text{where}  \ \ \ \
\psi_{\mathbf{p}, \mathbf{q}}=\sum_{\scriptsize{\begin{array}{c}
\mathbf{a} \subseteq \mathbf{p}, \mathbf{b}\subseteq\mathbf{q}\\
\mathbf{c}=\mathbf{p}\setminus \mathbf{a} \\
\mathbf{d}=\mathbf{q}\setminus \mathbf{b}
 \end{array}}}
 \phi_{\mathbf{a}, \mathbf{b}}\cdot\varphi_{\mathbf{c}, \mathbf{d}}\, .
\end{equation}

For each $\mathbf{p}, \mathbf{q}$, Proposition~\ref{P01}(2), Remark~\ref{R40} and (\ref{e46})
imply the following $R$-module isomorphism
\begin{equation}\label{e42}
B^{(k)}_{\mathbf{p}, \mathbf{q}}\cong \left\{ \begin{array}{l}
B_{\{k-p_r,\cdots, k-p_1, q_1+1, \cdots, q_s+1\}}\oplus B_{\{k-p_r, \cdots, k-p_1, k, q_1+1,\cdots, q_s+1\}}, \\
\mathrm{or}    \\
B_{\{1, k-p_r, \cdots, k-p_1, q_2+1,\cdots, q_s+1\}}\oplus B_{\{1, k-p_r, \cdots, k-p_1, k, q_2+1,\cdots, q_s+1\}}.
\end{array} \right.
\end{equation}
Therefore one gets an $R$-module isomorphism $Q_k\cong \oplus_{\mathbf{i}\subseteq [n]} B_{\mathbf{i}}$. Set $Q=Q_k$.

Assume that $\mathbf{p}, \mathbf{p}'\subseteq \{1, \cdots, k-2\}$ and $\mathbf{q}, \mathbf{q}'\subseteq \{k-1, \cdots, n-1\}$ such that
$\mathbf{p}\cap\mathbf{p}'=\emptyset$ and $\mathbf{q}\cap\mathbf{q}'=\emptyset$.
By Proposition~\ref{P01}(5) and Remark~\ref{R40}, the $R_k$-module
$B^{(k)}_{\mathbf{p}, \mathbf{q}}\otimes_{R_k}B^{(k)}_{\mathbf{p}', \mathbf{q}'}$ is a semidualizing and so
$B^{(k)}_{\mathbf{p}, \mathbf{q}}\otimes_{R_k}B^{(k)}_{\mathbf{p}', \mathbf{q}'}=
B^{(k)}_{\mathbf{p}\cup\mathbf{p}' , \mathbf{q}\cup \mathbf{q}'}$.
If $\phi_{\mathbf{p}, \mathbf{q}}\in B^{(k)}_{\mathbf{p}, \mathbf{q}}$ and
$\varphi_{\mathbf{p}', \mathbf{q}'}\in B^{(k)}_{\mathbf{p}', \mathbf{q}'}$, then
by  the isomorphism (\ref{e42}), one has
$\phi_{\mathbf{p}, \mathbf{q}}=(\beta_{\mathbf{p}, \mathbf{q}}, \gamma_{\mathbf{p}, \mathbf{q}})$ and
$\varphi_{\mathbf{p}', \mathbf{q}'}=(\beta_{\mathbf{p}', \mathbf{q}'}, \gamma_{\mathbf{p}', \mathbf{q}'})$, so that
$$\phi_{\mathbf{p}, \mathbf{q}} \cdot \varphi_{\mathbf{p}', \mathbf{q}'}=(\beta_{\mathbf{p}, \mathbf{q}}\cdot \beta_{\mathbf{p}', \mathbf{q}'},\  \beta_{\mathbf{p}, \mathbf{q}}\cdot \gamma_{\mathbf{p}', \mathbf{q}'} + \beta_{\mathbf{p}', \mathbf{q}'}\cdot\gamma_{\mathbf{p}, \mathbf{q}}) .$$
Thus by means of the ring structure on $Q_k$, (\ref{e45}), one can see that the resulting ring structure on $Q$ is as claimed.

The next step is to introduce the ideals $I_1, \cdots, I_n$. We set
$I_l = (\underbrace{0\oplus\cdots\oplus 0}_{2^{n-1}})\oplus (\oplus_{\mathbf{i}\subseteq [n],\, l\in\mathbf{i}}B_{\mathbf{i}})$, $1\leqslant l\leqslant n$,
which is an ideal of $Q$.
Also we have the following sequence of $R$-isomorphisms which preserve ring isomorphisms:
$$\begin{array}{llll}
Q/{(I_1+\cdots+I_n)}& = & (\oplus_{\mathbf{i}\subseteq [n]} B_{\mathbf{i}})/(\Sigma_{l=1}^n (\underbrace{0\oplus\cdots\oplus 0}_{2^{n-1}})\oplus (\oplus_{\mathbf{i}\subseteq [n],\, l\in\mathbf{i}}B_{\mathbf{i}}) )\\
&  \cong & (\oplus_{\mathbf{i}\subseteq [n]} B_{\mathbf{i}})/(\oplus_{\mathbf{i}\subseteq [n], \mathbf{i}\neq \emptyset} B_{\mathbf{i}})\\
                              &  \cong & R\, .
\end{array}$$

Note that each ideal
$I_{k,l}$, $1\leqslant l\leqslant n-1$, of $Q_k$ has the form
$I_{k,l}=(\underbrace{0\oplus\cdots\oplus 0}_{2^{n-2}})\oplus
(\oplus_{l\in\mathbf{p}\cup \mathbf{q}}B^{(k)}_{\mathbf{p}, \mathbf{q}})$.
Then, by (\ref{e42}), one has the following $R$-module isomorphism
$$I_{k,l}\cong \left\{ \begin{array}{ll}
I_{k-l} &   \text{if}\ \ 1\leqslant l\leqslant k-1\\
I_{l+1} &  \text{if}\ \ k\leqslant l\leqslant n-1.
\end{array} \right.$$
Also, by means of the ring isomorphism $Q_k\rightarrow Q$, we have the natural correspondence
between ideals:
$$I_{k,l}\stackrel{\text{correspond}}{\longleftarrow\hspace{-0.2cm}-\hspace{-0.2cm}-\hspace{-0.2cm}-\hspace{-0.2cm}-\hspace{-0.2cm}-
\hspace{-0.2cm}\longrightarrow} \left\{ \begin{array}{ll}
I_{k-l} &  \text{if}\ \ 1\leqslant l\leqslant k-1\\
I_{l+1} &  \text{if}\ \ k\leqslant l\leqslant n-1.
\end{array} \right.$$
Therefore for each $\Lambda\subseteq [n]\setminus\{k\}$,
there is a ring isomorphism
$Q/(\Sigma_{l\in \Lambda} I_l) \cong  Q_k/(\Sigma_{l\in\Lambda'} I_{k,l})$,
for some $\Lambda'\subseteq[n-1]$.
\end{cons}
The proof of Theorem~\ref{T42}, which is inspired by the proof of \cite[Theorem 3.2]{jls-w},
is rather technical and needs some preparatory lemmas.
\begin{lem}\label{L43}
Assume that $\Lambda\subseteq [n]$. Under the hypothesis of Theorem \ref{T42}, if $[n]\setminus \Lambda=\{b_1, \cdots, b_t\}$ with $1\leqslant b_1<\cdots<b_t\leqslant n$,
then there is an $R$-isomorphism
$$R_{_\Lambda}\cong \oplus_{\mathbf{i}\subseteq \{b_1, \cdots, b_t\}} B_{\mathbf{i}}$$
which induces a ring structure on $R_{_\Lambda}$ as follows.

For elements  $\big(\alpha_{\mathbf{i}}\big)_{\mathbf{i}\subseteq \{b_1, \cdots, b_t\}}$ and $\big(\theta_{\mathbf{i}}\big)_{\mathbf{i}\subseteq \{b_1, \cdots, b_t\}}$
 of $R_{_\Lambda}$
 $$\big(\alpha_{\mathbf{i}}\big)_{\mathbf{i}\subseteq \{b_1, \cdots, b_t\}}
\big(\theta_{\mathbf{i}}\big)_{\mathbf{i}\subseteq \{b_1, \cdots, b_t\}}=
\big(\sigma_{\mathbf{i}}\big)_{\mathbf{i}\subseteq \{b_1, \cdots, b_t\}}\  , \  where\ \ \ \ \
\sigma_{\mathbf{i}}=\sum_{\scriptsize{\begin{array}{c}
\mathbf{v}\subseteq \mathbf{i}\\
 \mathbf{w}=\mathbf{i}\setminus\mathbf{v} \end{array}} }
 \alpha_{\mathbf{v}}\cdot\theta_{\mathbf{w}}\, .$$
\end{lem}
\begin{proof}
We prove by induction on $n$. The case $n=1$ is clear. The case $n=2$ is proved in \cite{jls-w}.
Assume that $n>2$ and the assertion holds true for $n-1$.

If $\Lambda=[n]$, there is nothing to prove. Suppose that $|\Lambda|\leqslant n-1$
then there exists $k\in [n]$ such that $\Lambda\subseteq [n]\setminus\{k\}$. Thus, by Construction~\ref{cons},  there exists a subset $\Lambda'$ of $[n-1]$ such that
$R_{_\Lambda}\cong Q_k/(\Sigma_{l\in \Lambda'} I_{k,l})$  as ring isomorphism.

Note that $|[n-1]\setminus \Lambda'|=t-1$. Set $[n-1]\setminus \Lambda'=\{d_1, \cdots, d_u, d_{u+1}, \cdots, d_{t-1}\}$ such that
$1\leqslant d_1<\cdots<d_u<k-1$ and $k-1\leqslant d_{u+1}<\cdots< d_{t-1}\leqslant n-1$.
Then by induction there exists an $R_k$-isomorphism
 $$Q_k/(\Sigma_{l\in \Lambda'} I_{k,l})\cong \bigoplus_{\scriptsize{\begin{array}{c}
\mathbf{p}\subseteq \{d_1, \cdots, d_u\}\\
\mathbf{q}\subseteq \{d_{u+1}, \cdots, d_{t-1}\} \end{array}}}
B_{\mathbf{p}, \mathbf{q}}^{(k)}.$$
Proceeding as Construction~\ref{cons}, there is an $R$-isomorphism
$$\Big(\bigoplus_{\scriptsize{\begin{array}{c}
\mathbf{p}\subseteq \{d_1, \cdots, d_u\}\\
\mathbf{q}\subseteq \{d_{u+1}, \cdots, d_{t-1}\} \end{array}}}
B_{\mathbf{p}, \mathbf{q}}^{(k)}\Big)\
 \cong \  \Big(\bigoplus_{\mathbf{i}\subseteq \{b_1, \cdots, b_t\}} B_{\mathbf{i}}\Big).$$
 Therefore one has an $R$-isomorphism
$R_{_\Lambda}\cong \oplus_{\mathbf{i}\subseteq \{b_1, \cdots, b_t\}} B_{\mathbf{i}}$.
Similar to Construction~\ref{cons},  $R_{_\Lambda}$ has the desired ring structure.
\end{proof}

\begin{lem}\label{L44}
Under the hypothesis of Theorem~\ref{T42}, if $\Gamma\subsetneq \Lambda\subseteq [n]$, then
$\mathrm{Ext}^{\geqslant1}_{R_{_{\Gamma}}} (R_{_{\Lambda}}, R_{_{\Gamma}})=0$ and
$\mathrm{Hom}_{R_{_{\Gamma}}} (R_{_{\Lambda}}, R_{_{\Gamma}})$ is a non-free semidualizing $R_{_{\Lambda}}$-module.
\end{lem}
\begin{proof}
The case $n=1$ is clear and the case $n=2$ is proved in
\cite[Lemma 3.8]{jls-w}.
Let $n>2$ and suppose that the assertion is settled for $n-1$.

First assume that $\Lambda=[n]$. Set $[n]\setminus \Gamma=\{a_1, \cdots, a_s\}$ with $1\leqslant a_1<\cdots<a_s\leqslant n$.
By Lemma~\ref{L43}, $R_{_{\Gamma}}\cong \oplus_{\mathbf{i}\subseteq \{a_1, \cdots, a_s\}} B_{\mathbf{i}}$.
By Proposition~\ref{P01}(4) and Remark~\ref{R40},
$[B_{\{a_1, \cdots, a_s\}}]\trianglelefteq [B_{\mathbf{i}}]$ and
$\mathrm{Hom}_R(B_{\mathbf{i}}, B_{\{a_1, \cdots, a_s\}}) \cong B_{\{a_1, \cdots, a_s\}\setminus\mathbf{i}}$,
for all $\mathbf{i}\subseteq\{a_1, \cdots, a_s\}$.
Therefore there are $R$-isomorphisms
$$\mathrm{Hom}_R (R_{_{\Gamma}}, B_{\{a_1, \cdots, a_s\}})\cong \mathrm{Hom}_R (\oplus_{\mathbf{i}\subseteq \{a_1, \cdots, a_s\}} B_{\mathbf{i}}, B_{\{a_1, \cdots, a_s\}})\cong \oplus_{\mathbf{i}\subseteq \{a_1, \cdots, a_s\}} B_{\mathbf{i}}\cong R_{_{\Gamma}}$$
and, for all $i\geqslant 1$,
$$\mathrm{Ext}^i_R(R_{_{\Gamma}}, B_{\{a_1, \cdots, a_s\}})\cong \mathrm{Ext}^i_R(\oplus_{\mathbf{i}\subseteq \{a_1, \cdots, a_s\}} B_{\mathbf{i}}, B_{\{a_1, \cdots, a_s\}})=0.$$
Let $\mathbf{E}$ be an injective resolution of $B_{\{a_1, \cdots, a_s\}}$ as  an $R$-module. Thus $\mathrm{Hom}_R (R_{_{\Gamma}}, \mathbf{E})$ is an injective resolution of $R_{_{\Gamma}}$ as an $R_{_{\Gamma}}$-module. Note that the composition of natural homomorphisms \\
$R\rightarrow R_{_{\Gamma}}\rightarrow R$ is the identity $\mathrm{id}_R$.
Therefore
$$\mathrm{Hom}_{R_{_{\Gamma}}}(R, \mathrm{Hom}_R (R_{_{\Gamma}}, \mathbf{E}))\cong \mathrm{Hom}_R(R\otimes_{R_{_{\Gamma}}}R_{_{\Gamma}},\mathbf{E})\cong  \mathrm{Hom}_R(R, \mathbf{E})\cong \mathbf{E}\,.$$
Hence
$$\begin{array}{llll}
\mathrm{Ext}^i_{R_{_{\Gamma}}} (R, R_{_{\Gamma}})  &  \cong & \mathrm{H}^i(\mathrm{Hom}_{R_{_{\Gamma}}}(R, \mathrm{Hom}_R (R_{_{\Gamma}}, \mathbf{E}))) \\
                             &  \cong & \mathrm{H}^i(\mathbf{E}) \\
                              &  \cong & \left\{ \begin{array}{ll}
0 &\text{if}\ \ i>0\\
B_{\{a_1, \cdots, a_s\}} &\text{if}\ \ i=0\,.
\end{array} \right.
\end{array}$$
As $\{a_1, \cdots, a_s\}\neq \emptyset$, the $R$-module $B_{\{a_1, \cdots, a_s\}}$ is a non-free semidualizing.

Now assume that $|\Lambda|\leqslant n-1$. There exist $k\in [n]$, and subsets $\Gamma'$, $\Lambda'$ of $[n-1]$ such that there are $R$-isomorphisms and ring isomorphisms
$R_{_{\Gamma}}\cong Q_k/(\Sigma_{l\in \Gamma'} I_{k,l})$ and $R_{_{\Lambda}}\cong Q_k/(\Sigma_{l\in \Lambda'} I_{k,l})$,
where $Q_k$ and $I_{k,l}$ are as in Construction~\ref{cons}.
 By induction we have
$$\mathrm{Ext}^{i}_{R_{_{\Gamma}}} (R_{_{\Lambda}}, R_{_{\Gamma}})\cong \mathrm{Ext}^{i}_{Q_k/(\Sigma_{l\in \Gamma'} I_{k,l})} (Q_k/(\Sigma_{l\in \Lambda'} I_{k,l}), Q_k/(\Sigma_{l\in \Gamma'} I_{k,l}))=0$$
for all $i\geqslant 1$, and
$$\mathrm{Hom}_{R_{_{\Gamma}}} (R_{_{\Lambda}}, R_{_{\Gamma}})\cong \mathrm{Hom}_{Q_k/(\Sigma_{l\in \Gamma'} I_{k,l})} (Q_k/(\Sigma_{l\in \Lambda'} I_{k,l}), Q_k/(\Sigma_{l\in \Gamma'} I_{k,l}))$$
 is a non-free semidualizing $Q_k/(\Sigma_{l\in \Lambda'} I_{k,l})$-module.
Then $\mathrm{Hom}_{R_{_{\Gamma}}} (R_{_{\Lambda}}, R_{_{\Gamma}})$ is a non-free semidualizing $R_{_{\Lambda}}$-module.
\end{proof}

\begin{lem}\label{L45}
Under the hypothesis of Theorem \ref{T42}, if ${\Lambda}$ and ${\Gamma}$
are two subsets of $[n]$, then
$\mathrm{Tor}^{R_{_{{\Lambda}\cup {\Gamma}}}}_{\geqslant 1}(R_{_{\Lambda}}, R_{_{\Gamma}})=0$. Moreover, there is an $R_{_{\Lambda}}$-algebra isomorphism
$R_{_{\Lambda}}\otimes_{R_{_{{\Lambda}\cup {\Gamma}}}} R_{_{\Gamma}}\cong R_{_{{\Lambda}\cap {\Gamma}}}$.
\end{lem}
\begin{proof}
We prove by induction. If $n=1$, there is nothing to prove. The case $n=2$ is proved in
\cite[Lemma 3.9]{jls-w}.
Let $n>2$ and suppose that the assertion holds true for $n-1$.
First assume that ${\Lambda}\cup {\Gamma}=[n]$ and set $[n]\setminus {\Lambda}=\{b_1, \cdots, b_t\}$, $[n]\setminus {\Gamma}=\{a_1, \cdots, a_s\}$.
Then $[n]\setminus ({\Lambda}\cap {\Gamma})=\{b_1, \cdots, b_t, a_1, \cdots, a_s\}$.
By Lemma~\ref{L43},
$R_{_{\Lambda}}\cong \oplus_{\mathbf{i}\subseteq \{b_1, \cdots, b_t\}} B_{\mathbf{i}}\, \ \mathrm{and}\  \ R_{_{\Gamma}}\cong \oplus_{\mathbf{u}\subseteq \{a_1, \cdots, a_s\}} B_{\mathbf{u}}$.

As $\{b_1, \cdots, b_t\}\cap \{a_1, \cdots, a_s\}=\emptyset$, for each
$\mathbf{i}\subseteq \{b_1, \cdots, b_t\}$ and $\mathbf{u}\subseteq \{a_1, \cdots, a_s\}$,
by Proposition~\ref{P01}(5) and Remark~\ref{R40}, one has
$B_{\mathbf{i}}\in \mathcal{A}_{B_{\mathbf{u}}}(R)$ and so
$\mathrm{Tor}^R_{\geqslant1}(B_{\mathbf{i}}, B_{\mathbf{u}})=0$.
Hence $\mathrm{Tor}^R_{\geqslant1}(R_{_{\Lambda}}, R_{_{\Gamma}})=0$.

By Proposition~\ref{P01}(5) and Remark~\ref{R40}, the $R$-module
$B_{\mathbf{i}}\otimes_R B_{\mathbf{u}}$ is semidualizing and so \\
$B_{\mathbf{i}}\otimes_R B_{\mathbf{u}}= B_{\mathbf{i}\cup\mathbf{u}}$.
Therefore one has the natural $R$-module isomorphism
$$\eta: R_{_{\Lambda}}\otimes_R R_{_{\Gamma}}\longrightarrow
R_{_{{\Lambda}\cap{\Gamma}}}\, ,  \ \ \ \
\eta\big( (\alpha_{\mathbf{i}})_{\mathbf{i}\subseteq \{b_1, \cdots, b_t\}}
\otimes(\theta_{\mathbf{u}})_{\mathbf{u}\subseteq \{a_1, \cdots, a_s\}}\big)
=\big(\alpha_{\mathbf{i}}.\theta_{\mathbf{u}}\big)_{\scriptsize{\begin{array}{l}
\mathbf{i}\subseteq \{b_1, \cdots, b_t\}\\
\mathbf{u}\subseteq \{a_1, \cdots, a_s\}
\end{array}}}$$
It is routine to check that $\eta$  is also a ring isomorphism.

On the other hand the natural maps
$$\zeta:R_{_{\Lambda}}\rightarrow R_{_{\Lambda}}\otimes_R R_{_{\Gamma}}\, ,  \ \ \ \
\zeta{\big(}(\alpha_{\mathbf{i}})_{\mathbf{i}\subseteq \{b_1, \cdots, b_t\}}{\big)}=(\alpha_{\mathbf{i}})_{\mathbf{i}\subseteq \{b_1, \cdots, b_t\}}
\otimes(\stackrel{\centerdot}{\theta}_{\mathbf{u}})_{\mathbf{u}\subseteq \{a_1, \cdots, a_s\}}$$
and
$$\varepsilon:R_{_{\Lambda}}\rightarrow R_{_{{\Lambda}\cap {\Gamma}}}\, ,  \ \ \ \
\varepsilon((\alpha_{\mathbf{i}})_{\mathbf{i}\subseteq \{b_1, \cdots, b_t\}}{\big)}=(\chi_{\mathbf{v}})_{\mathbf{v}\subseteq \{a_1,\cdots, a_s, b_1, \cdots, b_t\}}\, ,
\hspace{1cm}$$

$$\mathrm{where}\ \ \ \ \ \ \ \ \ \ \ \ \
\stackrel{\centerdot}{\theta}_{\mathbf{u}}= \left\{ \begin{array}{ll}
0 &\text{if}\ \ \mathbf{u}\neq \emptyset\\
1 &\text{if}\ \ \mathbf{u}=\emptyset
\end{array} \right.   \ \ \ \ \ \   \mathrm{and}\ \ \ \ \ \ \
\chi_{\mathbf{v}}=\left\{ \begin{array}{ll}
\alpha_{\mathbf{v}}& \text{if}\ \  \mathbf{v}\cap \{a_1, \cdots, a_s\}=\emptyset\\
0 & \text{if}\ \  \mathbf{v}\cap \{a_1, \cdots, a_s\}\neq\emptyset\, ,
\end{array} \right.  $$
are ring homomorphism.
It is easy to check that  $\eta\zeta=\varepsilon$. Hence
$R_{_{\Lambda}}\otimes_R R_{_{\Gamma}}\stackrel{\eta}{\longrightarrow} R_{_{{\Lambda}\cap {\Gamma}}}$
is an $R_{_{\Lambda}}$-algebra isomorphism.

Now let ${\Lambda}\cup {\Gamma}\subsetneq [n]$, then, by Construction~\ref{cons},
there exist $k\in [n]$ and  ${\Lambda}', {\Gamma}'\subseteq[n-1]$
such that there are $R$-isomorphisms and ring isomorphisms
$$R_{_{\Lambda}}\cong Q_k/(\Sigma_{l\in {\Lambda}'} I_{k, l})\, , \ \ \ \ \ \ \ \ \
R_{_{\Gamma}}\cong Q_k/(\Sigma_{l\in {\Gamma}'} I_{k, l})$$
$$R_{_{{\Lambda}\cup {\Gamma}}}\cong Q_k/(\Sigma_{l\in {\Lambda}'\cup {\Gamma}'} I_{k, l})\, ,\ \mathrm{and}\ \ \  R_{_{{\Lambda}\cap {\Gamma}}}\cong Q_k/(\Sigma_{l\in {\Lambda}'\cap {\Gamma}'} I_{k, l})\,.$$
Thus, by induction, for all $i\geqslant 1$
$$\mathrm{Tor}^{R_{_{{\Lambda}\cup {\Gamma}}}}_{i}(R_{_{\Lambda}}, R_{_{\Gamma}})\cong \mathrm{Tor}^{Q_k/(\Sigma_{l\in {\Lambda}'\cup {\Gamma}'} I_{k,l})}_{i}(Q_k/(\Sigma_{l\in {\Lambda}'} I_{k,l}), Q_k/(\Sigma_{l\in {\Gamma}'} I_{k,l}))=0$$
and there is $Q_k/(\Sigma_{l\in {\Lambda}'} I_{k, l})$-algebra isomorphism and so $R_{_{\Lambda}}$-algebra isomorphism as follows
$$\begin{array}{llll}
R_{_{\Lambda}}\otimes_{R_{_{{\Lambda}\cup {\Gamma}}}} R_{_{\Gamma}}  &  \cong & Q_k/(\Sigma_{l\in {\Lambda}'} I_{k,l})\otimes_{Q_k/(\Sigma_{l\in {\Lambda}'\cup {\Gamma}'} I_{k, l})} Q_k/(\Sigma_{l\in {\Gamma}'} I_{k, l}) \\
                             &  \cong & Q_k/(\Sigma_{l\in {\Lambda}'\cap {\Gamma}'} I_{k, l}) \\
                              &  \cong & R_{_{{\Lambda}\cap {\Gamma}}}\,.
\end{array}$$
\end{proof}

\begin{lem}\label{L46}
Under the hypothesis of Theorem \ref{T42}, if ${\Lambda}$ and $\Gamma$ are two subsets of $[n]$, then $\mathrm{Tor}^{R_{_{\Lambda}}}_{\geqslant 1}(R_{_{{\Lambda}\cup {\Gamma}}}, R_{_{{\Lambda}\cap {\Gamma}}})=0$. Moreover, there is an $R_{_{{\Lambda}\cap {\Gamma}}}$-module isomorphism
$R_{_{{\Lambda}\cup {\Gamma}}}\otimes_{R_{_{\Lambda}}} R_{_{{\Lambda}\cap {\Gamma}}}\cong R_{_{\Gamma}}$.
\end{lem}
\begin{proof}
It is proved by induction on $n$. If $n=1$, there is nothing to prove. The case $n=2$ is proved in \cite[Lemma 3.11]{jls-w}.
Let $n>2$ and suppose that the assertion holds true for $n-1$.

First assume that ${\Lambda}\cup {\Gamma}=[n]$.
Let $\mathbf{P}$ be an $R$-projective resolution of $R_{_{\Gamma}}$. Lemma~\ref{L45} implies that
$R_{_{\Lambda}}\otimes_R \mathbf{P}$ is an $R_{_{\Lambda}}$-projective resolution of
$R_{_{\Lambda}}\otimes_R R_{_{\Gamma}}\cong R_{_{{\Lambda}\cap {\Gamma}}}$.
One has the following natural isomorphisms
$$R\otimes_{R_{_{\Lambda}}}(R_{_{\Lambda}}\otimes_R\mathbf{P})\cong (R\otimes_{R_{_{\Lambda}}}R_{_{\Lambda}})\otimes_R\mathbf{P}\cong R\otimes_R \mathbf{P} \cong \mathbf{P}$$
and then, for all $i\geqslant 1$,
$$\hspace{-0.5cm}\mathrm{Tor}^{R_{_{\Lambda}}}_{i}(R, R_{_{{\Lambda}\cap {\Gamma}}})   \cong
\mathrm{H}_i(R\otimes_{R_{_{\Lambda}}}(R_{_{\Lambda}}\otimes_R\mathbf{P})) \cong  \mathrm{H}_i(\mathbf{P}) =0\, .$$
Set $[n]\setminus {\Lambda}=\{b_1, \cdots, b_t\}$  and
$[n]\setminus {\Gamma}=\{a_1, \cdots, a_s\}$.
Then
$[n]\setminus ({\Lambda}\cap {\Gamma})=\{b_1, \cdots, b_t, a_1, \cdots, a_s\}$.
Consider the $R$-module isomorphism
$\xi : R_{_{\Gamma}}\stackrel{\cong}{\longrightarrow}  R\otimes_{R_{_{\Lambda}}}R_{_{{\Lambda}\cap {\Gamma}}}$
which is the composition
$$R_{_{\Gamma}}\stackrel{\cong}{-\hspace{-0.2cm}\longrightarrow} R\otimes_R R_{_{\Gamma}}    \stackrel{\cong}{-\hspace{-0.2cm}\longrightarrow}
R\otimes_{R_{_{\Lambda}}}(R_{_{\Lambda}}\otimes_R \, R_{_{\Gamma}}) \stackrel{\cong}{\underset{R\otimes \eta}{-\hspace{-0.2cm}-\hspace{-0.2cm}-\hspace{-0.2cm}\longrightarrow}}   R\otimes_{R_{_{\Lambda}}}R_{_{{\Lambda}\cap {\Gamma}}}$$
 given by
$$\begin{array}{llll} (\theta_{\mathbf{u}})_{\mathbf{u}\subseteq \{a_1, \cdots, a_s\}}\ \mapsto
\ 1\otimes (\theta_{\mathbf{u}})_{\mathbf{u}\subseteq \{a_1, \cdots, a_s\}}&\mapsto& 1\otimes[(\stackrel{\centerdot}{\alpha}_{\mathbf{i}})_{\mathbf{i}\subseteq \{b_1, \cdots, b_t\}}\otimes(\theta_{\mathbf{u}})_{\mathbf{u}\subseteq \{a_1, \cdots, a_s\}}]\\
&\mapsto& 1\otimes (\lambda_{\mathbf{v}})_{\mathbf{v}\subseteq \{a_1, \cdots, a_s, b_1, \cdots, b_t\}}\, ,
\end{array}$$
$$\mathrm{where} \hspace{1.2cm} \stackrel{\centerdot}{\alpha}_{\mathbf{i}}= \left\{ \begin{array}{ll}
0 &\text{if}\ \ \mathbf{i}\neq \emptyset\\
1 & \text{if}\ \ \mathbf{i}=\emptyset
\end{array} \right.  \ \ \  \ \ \ \ \   \mathrm{and}\ \ \  \ \ \ \ \ \
\lambda_{\mathbf{v}}=\left\{ \begin{array}{ll}
\theta_{\mathbf{v}}&\text{if}\ \  \mathbf{v}\cap \{b_1, \cdots, b_t\}=\emptyset\\
0 &\text{if}\ \ \mathbf{v}\cap \{b_1, \cdots, b_t\}\neq\emptyset
\end{array} \right. . \hspace{2.2cm} $$
We claim that $\xi$ is an $R_{_{{\Lambda}\cap {\Gamma}}}$-module isomorphism.

\textit{Proof of the claim:}
The $R_{_{{\Lambda}\cap {\Gamma}}}$-module structure of $R_{_{\Gamma}}$, which is given via the natural surjection
$R_{_{{\Lambda}\cap {\Gamma}}}\rightarrow R_{_{\Gamma}}$, is described as
$$(\gamma_{\mathbf{v}})_{\mathbf{v}\subseteq \{a_1, \cdots, a_s, b_1, \cdots, b_t\}}(\theta_{\mathbf{u}})_{\mathbf{u}\subseteq \{a_1, \cdots, a_s\}}= (\gamma_{\mathbf{u}})_{\mathbf{u}\subseteq \{a_1, \cdots, a_s\}}(\theta_{\mathbf{u}})_{\mathbf{u}\subseteq \{a_1, \cdots, a_s\}},$$
where $(\gamma_{\mathbf{v}})_{\mathbf{v}\subseteq \{a_1, \cdots, a_s, b_1, \cdots, b_t\}}$ is an element of $R_{_{{\Lambda}\cap {\Gamma}}}$.
In the following we check that
$$\xi\big((\gamma_{\mathbf{v}})_{\mathbf{v}\subseteq \{a_1, \cdots, a_s, b_1, \cdots, b_t\}}(\theta_{\mathbf{u}})_{\mathbf{u}\subseteq \{a_1, \cdots, a_s\}}\big)=
(\gamma_{\mathbf{v}})_{\mathbf{v}\subseteq \{a_1, \cdots, a_s, b_1, \cdots, b_t\}}[\xi((\theta_{\mathbf{u}})_{\mathbf{u}\subseteq \{a_1, \cdots, a_s\}})].$$
Note that
$$\begin{array}{llll}
\xi\big((\gamma_{\mathbf{v}})_{\mathbf{v}\subseteq \{a_1, \cdots, a_s, b_1, \cdots, b_t\}}(\theta_{\mathbf{u}})_{\mathbf{u}\subseteq \{a_1, \cdots, a_s\}}\big)  &  = & \xi((\gamma_{\mathbf{u}})_{\mathbf{u}\subseteq \{a_1, \cdots, a_s\}}(\theta_{\mathbf{u}})_{\mathbf{u}\subseteq \{a_1, \cdots, a_s\}}) \\
 & = & \xi\big( (\sigma_{\mathbf{u}})_{\mathbf{u}\subseteq \{a_1, \cdots, a_s\}}\big) \\
                             &   =  & 1\otimes (\mu_{\mathbf{v}})_{\mathbf{v}\subseteq \{a_1, \cdots, a_s, b_1, \cdots, b_t\}},
\end{array}$$
where
$(\sigma_{\mathbf{u}})_{\mathbf{u}\subseteq \{a_1, \cdots, a_s\}}= (\gamma_{\mathbf{u}})_{\mathbf{u}\subseteq \{a_1, \cdots, a_s\}}(\theta_{\mathbf{u}})_{\mathbf{u}\subseteq \{a_1, \cdots, a_s\}} \  \mathrm{and}\ \mu_{\mathbf{v}}= \left\{ \begin{array}{ll}
\sigma_{\mathbf{v}} &\text{if}\ \ \mathbf{v}\cap \{b_1, \cdots, b_t\}=\emptyset\\
0 &\text{if}\ \ \mathbf{v}\cap \{b_1, \cdots, b_t\}\neq\emptyset
\end{array} \right.. $
On the other hand
$$\begin{array}{llll}
(\gamma_{\mathbf{v}})_{\mathbf{v}\subseteq \{a_1, \cdots, a_s, b_1, \cdots, b_t\}}[\xi((\theta_{\mathbf{u}})_{\mathbf{u}\subseteq \{a_1, \cdots, a_s\}})]&\hspace{-0.15cm} = &\hspace{-0.2cm}(\gamma_{\mathbf{v}})_{\mathbf{v}\subseteq \{a_1, \cdots, a_s, b_1, \cdots, b_t\}}[1\otimes (\lambda_{\mathbf{v}})_{\mathbf{v}\subseteq \{a_1, \cdots, a_s, b_1, \cdots, b_t\}}]\\
    &\hspace{-0.15cm} = &\hspace{-0.2cm} 1\otimes [(\gamma_{\mathbf{v}})_{\mathbf{v}\subseteq \{a_1, \cdots, a_s, b_1, \cdots, b_t\}}(\lambda_{\mathbf{v}})_{\mathbf{v}\subseteq \{a_1, \cdots, a_s, b_1, \cdots, b_t\}}]\\
      &\hspace{-0.15cm} = &\hspace{-0.2cm} 1\otimes (\varrho_{\mathbf{v}})_{\mathbf{v}\subseteq \{a_1, \cdots, a_s, b_1, \cdots, b_t\}}\\
      &\hspace{-0.15cm} = &\hspace{-0.2cm} [1\otimes (\mu_{\mathbf{v}})_{\mathbf{v}\subseteq \{a_1, \cdots, a_s, b_1, \cdots, b_t\}}] + [1\otimes \delta],
\end{array}$$
where
$\delta=(\delta_{\mathbf{v}})_{\mathbf{v}\subseteq \{a_1, \cdots, a_s, b_1, \cdots, b_t\}}$ with
$\delta_{\mathbf{v}}=\left\{ \begin{array}{ll}
0  & \text{if}\ \  \mathbf{v}\cap \{b_1, \cdots, b_t\}=\emptyset\\
\varrho_{\mathbf{v}} & \text{if}\ \  \mathbf{v}\cap \{b_1, \cdots, b_t\}\neq\emptyset
\end{array}\right. .$

It is enough to show that  $1\otimes \delta=0$.
To this end, we have
$$1\otimes \delta=
\sum_{\scriptsize{\begin{array}{c}
\mathbf{w}\subseteq \{a_1, \cdots, a_s, b_1, \cdots, b_t\}\\
 \mathbf{w}\cap \{b_1, \cdots, b_t\}\neq \emptyset
 \end{array}}}
1\otimes \delta({\tiny{\mathbf{w}}}),$$
where
$\delta({\tiny{\mathbf{w}}})=\big(\delta({\tiny{\mathbf{w}}})_{\mathbf{v}} \big)_{\mathbf{v}\subseteq
\{a_1, \cdots, a_s, b_1, \cdots, b_t\}}$
with
$\delta({\tiny{\mathbf{w}}})_{\mathbf{v}}=\left\{ \begin{array}{ll}
0  & \text{if}\ \  \mathbf{v}\neq  \mathbf{w}\\
\delta_{\mathbf{w}} & \text{if}\ \  \mathbf{v}= \mathbf{w}
\end{array}\right. .$
For each $\mathbf{w}$ there exist $\mathbf{w}'\subseteq \{ b_1, \cdots, b_t\}$ and $\mathbf{w}''\subseteq\{a_1, \cdots, a_s\}$ with
$\mathbf{w}'\cup \mathbf{w}''=\mathbf{w}$. Thus $B_{\mathbf{w}'}\otimes_R B_{\mathbf{w}''}\stackrel{\rho_{\mathbf{w}}}{\cong} B_{\mathbf{w}}$ and there exist
$\delta'_{\mathbf{w}}\in B_{\mathbf{w}'}$ and  $\delta''_{\mathbf{w}}\in B_{\mathbf{w}''}$
such that
$\delta_{\mathbf{w}}=\rho_{\mathbf{w}}(\delta'_{\mathbf{w}}\otimes \delta''_{\mathbf{w}})$.

Set
$\alpha(\tiny{\mathbf{w}})=\big(\alpha(\tiny{\mathbf{w}})_{\mathbf{i}}\big)_{\mathbf{i}
\subseteq \{ b_1, \cdots, b_t\}}$, where
$\alpha({\tiny{\mathbf{w}}})_{\mathbf{i}}=\left\{\begin{array}{ll}
0  & \text{if}\ \  \mathbf{i}\neq \mathbf{w}'\\
\delta'_{\mathbf{w}} & \text{if}\ \  \tiny{\mathbf{i}= \mathbf{w}'}
\end{array} \right.$.
As the $R_{_\Lambda}$-module structure on $R$ is given via the natural surjection $R_{_\Lambda}\longrightarrow R$, and $\alpha(\tiny{\mathbf{w}})$ is an element
of the kernel of this map,
$0\oplus(\oplus_{\mathbf{i}\subseteq \{ b_1, \cdots, b_t\}, \mathbf{i}\neq \emptyset} B_{\mathbf{i}})$,
we have $1\,\alpha(\tiny{\mathbf{w}})=0$.
Set
$\beta(\tiny{\mathbf{w}})=\big( \beta(\tiny{\mathbf{w}})_{\mathbf{v}}\big)_{\mathbf{v}\subseteq \{a_1, \cdots, a_s, b_1, \cdots, b_t\}}$, where
$\beta({\tiny{\mathbf{w}}})_{\mathbf{v}}=\left\{\begin{array}{ll}
0 & \text{if}\ \  \mathbf{v}\neq \mathbf{w}''\\
\delta''_{\mathbf{w}}  & \text{if}\ \  \tiny{\mathbf{v}=\mathbf{w}''}
\end{array} \right.\!.$
Note that $\beta(\tiny{\mathbf{w}})$ is an element of $R_{_{{\Lambda}\cap {\Gamma}}}$ and
$\delta({\tiny{\mathbf{w}}})=\alpha(\tiny{\mathbf{w}})\beta(\tiny{\mathbf{w}})$.
Then
$$1\otimes \delta=\sum_{\mathbf{w}} 1\otimes \delta({\tiny{\mathbf{w}}})=
\sum_{\mathbf{w}} 1\otimes[\alpha(\tiny{\mathbf{w}}) \beta(\tiny{\mathbf{w}})]
= \sum_{\mathbf{w}} [1\alpha(\tiny{\mathbf{w}})]\otimes
\beta(\tiny{\mathbf{w}})
=\sum_{\mathbf{w}} 0\otimes\beta(\tiny{\mathbf{w}})=0.$$
Therefore the claim is proved and also the assertion holds in the case
${\Lambda}\cup {\Gamma}= [n]$.

\hspace{-0.3cm}We treat the case ${\Lambda}\cup {\Gamma}\subsetneq [n]$ by induction and its details are similar to the proof of Lemma~\ref{L45}.
\end{proof}

\vspace{0.2cm}
\emph{Proof of Theorem \ref{T42}}. \  (1) is proved in Construction~\ref{cons}.

(2). It is proved by induction on $n$. The case $n=1$ is clear by assumptions.
Let $n>1$ and the claim is settled for $n-1$.
If $\Lambda=[n]$, then $R_{_\Lambda}\cong R$  and is Cohen-Macaulay with the dualizing module $D$ and is not Gorenstein.
Let $\Lambda\subsetneq [n]$.  There exists $k\in [n]$ such that
$\Lambda\subseteq [n]\setminus \{k\}$. By Construction~\ref{cons}, there exists
a subset $\Lambda'\neq \emptyset$ of $[n-1]$ such that
$R_{_\Lambda}\cong Q_k/(\Sigma_{l\in \Lambda'} I_{k,l})$ as ring isomorphism. Thus, by induction, $R_{_\Lambda}$ is non-Gorenstein Cohen-Macaulay ring with dualizing module.

(3). It is clear that $\prod_{l\in \Lambda} I_l\subseteq \bigcap_{l\in \Lambda}I_l$\,. Let $\alpha=\big(\alpha_{\mathbf{i}}\big)_{\mathbf{i}\subseteq [n]}$ be an element of $\bigcap_{l\in \Lambda}I_l$\,.
Then, by Construction~\ref{cons}, $\alpha_{\mathbf{i}}=0$ for all $\mathbf{i}\subseteq [n]$ with $\Lambda\nsubseteq \mathbf{i}$\,. We have
$\alpha=\sum_{\Lambda\subseteq \mathbf{v}\subseteq[n]}\alpha(\tiny{\mathbf{v}})$,
where
$\alpha({\tiny{\mathbf{v}}})=\big(\alpha({\tiny{\mathbf{v}}})_{\mathbf{i}} \big)_{\mathbf{i}\subseteq [n]}$
with
$\alpha({\tiny{\mathbf{v}}})_{\mathbf{i}}=\left\{ \begin{array}{ll}
0  &\text{if}\ \  \mathbf{i}\neq  \mathbf{v}\\
\alpha_{\mathbf{v}} &\text{if}\ \  \mathbf{i}= \mathbf{v}
\end{array}\right.$.  Set $\Lambda=\{a_1, \cdots, a_m\}$.
If $\mathbf{v}\subseteq [n]$ such that
$\Lambda\subseteq \mathbf{v}$, then
$\mathbf{v}=\{a_1\}\cup\{a_2\}\cup\cdots \cup\{a_{m-1}\}\cup(\mathbf{v}\setminus\{a_1, \cdots, a_{m-1}\})$.
Thus
$$B_{\mathbf{v}}\stackrel{\Phi}{\cong }B_{\{a_1\}}\otimes_R \cdots \otimes_R B_{\{a_{m-1}\}}\otimes_R B_{\mathbf{v}\setminus\{a_1, \cdots, a_{m-1}\}}.$$
Therefore there exist $\theta_{\mathbf{v}, m}\in B_{\mathbf{v}\setminus\{a_1, \cdots, a_{m-1}\}}$ and $\theta_{\mathbf{v}, r}\in B_{\{a_r\}}$, $1\leqslant r< m$, such that
$\alpha_{\mathbf{v}}=\Phi(\theta_{\mathbf{v}, 1}\otimes\cdots \otimes\theta_{\mathbf{v}, m-1}\otimes \theta_{\mathbf{v}, m})$. Set
$\varphi({\tiny{\mathbf{v}}}, r)=\big(\varphi({\tiny{\mathbf{v}}}, r)_{\mathbf{i}}\big)_{\mathbf{i}\subseteq [n]}$, $1\leqslant r\leqslant m$, where, for
$1\leqslant r < m$,
 \vspace{0.2cm}

\hspace{-0.4cm}$\varphi({\tiny{\mathbf{v}}}, r)_{\mathbf{i}}=\left\{\begin{array}{ll}
0  & \text{if}\ \ \mathbf{i}\neq \{a_r\}\\
\theta_{\mathbf{v}, r} & \text{if}\ \ \mathbf{i}=\{a_r\}
\end{array} \right.$
and \
$\varphi({\tiny{\mathbf{v}}}, m)_{\mathbf{i}}=\left\{\begin{array}{ll}
0  & \text{if}\  \  \mathbf{i}\neq \mathbf{v}\setminus \{a_1, \cdots, a_{m-1}\}\\
\theta_{\mathbf{v}, m} & \text{if}\ \ \mathbf{i}=\mathbf{v}\setminus \{a_1, \cdots, a_{m-1}\}\,.
\end{array} \right.$
\vspace{0.2cm}

Note that $\varphi({\tiny{\mathbf{v}}}, r)\in I_{a_r}$, $1\leqslant r\leqslant m$.
Hence
$\varphi({\tiny{\mathbf{v}}}, 1)\cdots \varphi({\tiny{\mathbf{v}}}, m-1)
\varphi({\tiny{\mathbf{v}}}, m)\in \prod_{l\in \Lambda} I_l$.
On the other hand
$\varphi({\tiny{\mathbf{v}}}, 1)\cdots \varphi({\tiny{\mathbf{v}}}, m-1)
\varphi({\tiny{\mathbf{v}}}, m)= \alpha(\tiny{\mathbf{v}})$. Thus
$\alpha(\tiny{\mathbf{v}})$ is an element of $\prod_{l\in \Lambda} I_l$ and so
$\alpha\in\prod_{l\in \Lambda} I_l$.

(4) is followed by Remark~\ref{R1} and Lemma~\ref{L44}.

(5). Let $\mathbf{P}$ be a projective resolution of $R_{_{\Lambda\cup \Gamma}}$ over $R_{_{\Lambda}}$.
Lemma~\ref{L46} implies that the complex
$\mathbf{P}\otimes_{R_{_{\Lambda}}} R_{_{{\Lambda}\cap {\Gamma}}}$
is a $R_{_{{\Lambda}\cap {\Gamma}}}$-projective resolution of
$R_{_{{\Lambda}\cup {\Gamma}}}\otimes_{R_{_{\Lambda}}} R_{_{{\Lambda}\cap {\Gamma}}}\cong R_{_{\Gamma}}$.
From the isomorphisms
$$(\mathbf{P}\otimes_{R_{_{\Lambda}}} R_{_{{\Lambda}\cap {\Gamma}}})\otimes_{R_{_{{\Lambda}\cap {\Gamma}}}}R_{_{\Lambda}}
    \cong   \mathbf{P}\otimes_{R_{_{\Lambda}}}R_{_{\Lambda}}\cong \mathbf{P}$$
one gets
$$\mathrm{Tor}^{R_{_{{\Lambda}\cap {\Gamma}}}}_{i}(R_{_{\Gamma}}, R_{_{\Lambda}})
\cong \mathrm{H}_i((\mathbf{P}\otimes_{R_{_{\Lambda}}} R_{_{{\Lambda}\cap {\Gamma}}})\otimes_{R_{_{{\Lambda}\cap {\Gamma}}}}R_{_{\Lambda}}) \cong
\mathrm{H}_i(\mathbf{P}) =0\, .$$
for all $i\geqslant 1$. There is a natural isomorphism
$R_{_{\Lambda}}\otimes_{R_{_{{\Lambda}\cap {\Gamma}}}}R_{_{\Gamma}}\cong R_{_{{\Lambda}\cup {\Gamma}}}$
which is both $R_{_{{\Lambda}\cap {\Gamma}}}$- and $R_{_{\Gamma}}$-isomorphism.

Let $\mathbf{P}'$ be an $R_{_{{\Lambda}\cap {\Gamma}}}$-projective resolution of $R_{_{\Lambda}}$. As seen in the above,
$\mathbf{P}'\otimes_{R_{_{{\Lambda}\cap {\Gamma}}}}R_{_{\Gamma}}$
is a projective resolution of $R_{_{{\Lambda}\cup {\Gamma}}}$ over
$R_{_{\Gamma}}$. Therefore we have
$$\begin{array}{lll}
\mathrm{Ext}^i_{R_{_{{\Lambda}\cap {\Gamma}}}}(R_{_{\Lambda}}, R_{_{\Gamma}})&\cong &
\mathrm{H}^i(\mathrm{Hom}_{R_{_{{\Lambda}\cap {\Gamma}}}}(\mathbf{P}', R_{_{\Gamma}})\\
 & \cong & \mathrm{H}^i(\mathrm{Hom}_{R_{_{\Gamma}}}(\mathbf{P}'\otimes_{R_{_{{\Lambda}\cap {\Gamma}}}}R_{_{\Gamma}}, R_{_{\Gamma}})\\
 & \cong & \mathrm{Ext}^i_{R_{_{\Gamma}}}(R_{_{{\Lambda}\cup {\Gamma}}}, R_{_{\Gamma}})
 \end{array}$$
for all $i\geqslant 0$. By (4), $\mbox{G}$-$\mathrm{dim}_{R_{_{\Gamma}}}R_{_{{\Lambda}\cup {\Gamma}}}=0$, and so one gets
$\mathrm{Ext}^{\geqslant 1}_{R_{_{{\Lambda}\cap {\Gamma}}}}(R_{_{\Lambda}}, R_{_{\Gamma}})=0.$
Also, by (4), $\mathrm{Hom}_{R_{_{\Gamma}}}(R_{_{{\Lambda}\cup {\Gamma}}}, R_{_{\Gamma}})$ is a non-free semidualizing
$R_{_{{\Lambda}\cup {\Gamma}}}$-module and thus $\mathrm{Hom}_{R_{_{{\Lambda}\cap {\Gamma}}}}(R_{_{\Lambda}}, R_{_{\Gamma}})$ is not cyclic.

(6).  As $R_{_{{\Lambda}\cap {\Gamma}}}=Q/(\Sigma_{l\in {\Lambda}\cap {\Gamma}} I_l) $ and
$R_{_{\Lambda}}=Q/(\Sigma_{l\in {\Lambda}} I_l)\cong R_{_{{\Lambda}\cap {\Gamma}}}/
({\Sigma_{l\in {\Lambda}} I_l/(\Sigma_{l\in {\Lambda}\cap {\Gamma}} I_l)})$,
one has the natural isomorphism
$$\kappa: \mathrm{Hom}_{R_{_{{\Lambda}\cap {\Gamma}}}}(R_{_{\Lambda}} , R_{_{{\Lambda}\cap {\Gamma}}})\longrightarrow\big(0 :_{R_{_{{\Lambda}\cap {\Gamma}}}} {\Sigma_{l\in {\Lambda}} I_l/(\Sigma_{l\in {\Lambda}\cap {\Gamma}} I_l)} \big), \ \
\kappa(\psi)=\psi(\stackrel{\centerdot}{\alpha}),$$
where
$\stackrel{\centerdot}{\alpha}=(\stackrel{\centerdot}{\alpha}_{\mathbf{i}})_{\mathbf{i}\subseteq [n]\setminus \Lambda}$
with
$\stackrel{\centerdot}{\alpha}_{\mathbf{i}}= \left\{ \begin{array}{ll}
0 &\text{if}\ \ \mathbf{i}\neq \emptyset\\
1 & \text{if}\ \ \mathbf{i}=\emptyset
\end{array} \right.$
is the identity element of $R_{_\Lambda}$.

Next we show that
$\big(0 :_{R_{_{{\Lambda}\cap {\Gamma}}}} {\Sigma_{l\in {\Lambda}} I_l/(\Sigma_{l\in {\Lambda}\cap {\Gamma}} I_l)} \big)= {\Sigma_{l\in {\Lambda}} I_l/(\Sigma_{l\in {\Lambda}\cap {\Gamma}} I_l)}$.
Set $\Lambda\setminus\Gamma=\{a\}$.
Let
$\gamma=(\gamma_{\mathbf{i}})_{\mathbf{i}\subseteq [n]\setminus{{\Lambda}\cap {\Gamma}}}$
be an element of
$\big(0 :_{R_{_{{\Lambda}\cap {\Gamma}}}} {\Sigma_{l\in {\Lambda}} I_l/(\Sigma_{l\in {\Lambda}\cap {\Gamma}} I_l)} \big)$.
If $\gamma\notin {\Sigma_{l\in {\Lambda}} I_l/(\Sigma_{l\in {\Lambda}\cap {\Gamma}} I_l)}$, then
there exists $\mathbf{v}\subseteq [n]\setminus{{\Lambda}\cap {\Gamma}}$ such that
$a\notin \mathbf{v}$ and $\gamma_{\mathbf{v}}\neq 0$.
Set $M=R\gamma_{\mathbf{v}}$, which is a non-zero submodule of $B_{\mathbf{v}}$.
As $B_a$ is a semidualizing $R$-module and $M\neq 0$, we have $B_a\otimes_R M\neq0$.
Thus there exists an element $e$ of $B_a$ such that $e\otimes\gamma_{\mathbf{v}}\neq0$.
Set $\theta=(\theta_{\mathbf{i}})_{\mathbf{i}\subseteq [n]\setminus{{\Lambda}\cap {\Gamma}}}$,
where \\
$\theta_{\mathbf{i}}= \left\{ \begin{array}{ll}
0 &\text{if}\ \ \mathbf{i}\neq \{a\}\\
e & \text{if}\ \ \mathbf{i}=\{a\}
\end{array} \right.$.
Note that $\theta$ is an element of
${\Sigma_{l\in {\Lambda}} I_l/(\Sigma_{l\in {\Lambda}\cap {\Gamma}} I_l)}$ and
$\gamma \theta \neq0$, which contradicts with
$\gamma\in \big(0 :_{R_{_{{\Lambda}\cap {\Gamma}}}} {\Sigma_{l\in {\Lambda}} I_l/(\Sigma_{l\in {\Lambda}\cap {\Gamma}} I_l)} \big)$.
Therefore
$$\big(0 :_{R_{_{{\Lambda}\cap {\Gamma}}}} {\Sigma_{l\in {\Lambda}} I_l/(\Sigma_{l\in {\Lambda}\cap {\Gamma}} I_l)} \big)\subseteq {\Sigma_{l\in {\Lambda}} I_l/(\Sigma_{l\in {\Lambda}\cap {\Gamma}} I_l)}.$$
On the other hand
${\Sigma_{l\in {\Lambda}} I_l/(\Sigma_{l\in {\Lambda}\cap {\Gamma}} I_l)}\subseteq
\big(0 :_{R_{_{{\Lambda}\cap {\Gamma}}}} {\Sigma_{l\in {\Lambda}} I_l/(\Sigma_{l\in {\Lambda}\cap {\Gamma}} I_l)} \big)$. Indeed, if
$\alpha=(\alpha_{\mathbf{i}})_{\mathbf{i}\subseteq [n]\setminus{{\Lambda}\cap {\Gamma}}}$
and
$\alpha'=(\alpha'_{\mathbf{i}})_{\mathbf{i}\subseteq [n]\setminus{{\Lambda}\cap {\Gamma}}}$
are two elements of ${\Sigma_{l\in {\Lambda}} I_l/(\Sigma_{l\in {\Lambda}\cap {\Gamma}} I_l)}$,
then $\alpha_{\mathbf{i}}=0=\alpha'_{\mathbf{i}}$ for all $\mathbf{i}$ such that  $a\notin\mathbf{i}$.
Hence, by Lemma~\ref{L43},  $\alpha\alpha'=0$.
Thus
\begin{equation}\label{e47}
\kappa: \mathrm{Hom}_{R_{_{{\Lambda}\cap {\Gamma}}}}(R_{_{\Lambda}} , R_{_{{\Lambda}\cap {\Gamma}}})\longrightarrow{\Sigma_{l\in {\Lambda}} I_l/(\Sigma_{l\in {\Lambda}\cap {\Gamma}} I_l)}, \ \
\kappa(\psi)=\psi(\stackrel{\centerdot}{\alpha})
\end{equation}
is an $R_{_{{\Lambda}\cap {\Gamma}}}$-isomorphism.

By (4), $\mbox{G}$-$\mathrm{dim}_{R_{_{{\Lambda}\cap {\Gamma}}}} R_{_{\Lambda}}=0$.
Let $\mathbf{F}$ be a minimal free resolution of $R_{_{\Lambda}}$ over $R_{_{{\Lambda}\cap {\Gamma}}}$. Note that ${\Sigma_{l\in {\Lambda}} I_l/(\Sigma_{l\in {\Lambda}\cap {\Gamma}} I_l)}$ is the first syzygy of $R_{_{\Lambda}}$ in $\mathbf{F}$. By \cite[Construction 3.6]{am} and
(\ref{e47}), we can construct  a Tate resolution of $R_{_{\Lambda}}$ as $\mathbf{T}\rightarrow \mathbf{F}\rightarrow R_{_{\Lambda}}$, where $\mathbf{T}$ construct by splicing $\mathbf{F}$ with
$\mathrm{Hom}_{R_{_{{\Lambda}\cap {\Gamma}}}}(\mathbf{F} , R_{_{{\Lambda}\cap {\Gamma}}})$. Hence
$\mathbf{T}\cong \mathrm{Hom}_{R_{_{{\Lambda}\cap {\Gamma}}}}(\mathbf{T} , R_{_{{\Lambda}\cap {\Gamma}}})$.
This explains the first isomorphism in the next sequence
\begin{equation}\label{e43}
\begin{array}{lll}
\widehat{\mathrm{Tor}}^{R_{_{{\Lambda}\cap {\Gamma}}}}_i(R_{_{\Lambda}}, R_{_{\Gamma}})& = & \mathrm{H}_i\big(\mathbf{T}\otimes_{R_{_{{\Lambda}\cap {\Gamma}}}} R_{_{\Gamma}}\big)\\
  & \cong & \mathrm{H}_i\big(\mathrm{Hom}_{R_{_{{\Lambda}\cap {\Gamma}}}}(\mathbf{T} , R_{_{{\Lambda}\cap {\Gamma}}})\otimes_{R_{_{{\Lambda}\cap {\Gamma}}}} R_{_{\Gamma}}\big)\\
  & \cong & \mathrm{H}_i\big(\mathrm{Hom}_{R_{_{{\Lambda}\cap {\Gamma}}}}(\mathbf{T} , R_{_{\Gamma}})\big)\\
  & = & \widehat{\mathrm{Ext}}^{-i}_{R_{_{{\Lambda}\cap {\Gamma}}}}(R_{_{\Lambda}}, R_{_{\Gamma}})
\end{array}
\end{equation}
for all $i\in \mathbb{Z}$.
As each $R_{_{{\Lambda}\cap {\Gamma}}}$-module $\mathbf{T}_i$ is finite and free,
the second isomorphism follows.

By (4), $\mbox{G}$-$\mathrm{dim}_{R_{_{{\Lambda}\cap {\Gamma}}}} R_{_{\Lambda}}=0$ and so,
by \cite[Theorem 5.2]{am}, one has
\begin{equation}\label{e44}
\widehat{\mathrm{Tor}}^{R_{_{{\Lambda}\cap {\Gamma}}}}_i(R_{_{\Lambda}}, R_{_{\Gamma}})
 \cong  \mathrm{Tor}^{R_{_{{\Lambda}\cap {\Gamma}}}}_i(R_{_{\Lambda}}, R_{_{\Gamma}})
 \ \ \ \ \mathrm{and}\ \ \ \ \
 \widehat{\mathrm{Ext}}^{i}_{R_{_{{\Lambda}\cap {\Gamma}}}}(R_{_{\Lambda}}, R_{_{\Gamma}}) \cong  \mathrm{Ext}^{i}_{R_{_{{\Lambda}\cap {\Gamma}}}}
 (R_{_{\Lambda}}, R_{_{\Gamma}})
\end{equation}
for all $i\geqslant 1$. Thus, by (\ref{e43}), (\ref{e44}) and (5), one gets
$$\widehat{\mathrm{Ext}}^{-i}_{R_{_{{\Lambda}\cap {\Gamma}}}}(R_{_{\Lambda}}, R_{_{\Gamma}})\cong \widehat{\mathrm{Tor}}^{R_{_{{\Lambda}\cap {\Gamma}}}}_i(R_{_{\Lambda}}, R_{_{\Gamma}})\cong  \mathrm{Tor}^{R_{_{{\Lambda}\cap {\Gamma}}}}_i(R_{_{\Lambda}}, R_{_{\Gamma}})=0\,,$$
$$\widehat{\mathrm{Tor}}^{R_{_{{\Lambda}\cap {\Gamma}}}}_{-i}(R_{_{\Lambda}}, R_{_{\Gamma}})  \cong  \widehat{\mathrm{Ext}}^{i}_{R_{_{{\Lambda}\cap {\Gamma}}}}(R_{_{\Lambda}}, R_{_{\Gamma}})
  \cong  \mathrm{Ext}^{i}_{R_{_{{\Lambda}\cap {\Gamma}}}}(R_{_{\Lambda}}, R_{_{\Gamma}})=0$$
for all $i\geqslant 1$. Therefore, by (\ref{e43}), to complete the proof it is enough to show that $\widehat{\mathrm{Ext}}^{0}_{R_{_{{\Lambda}\cap {\Gamma}}}}(R_{_{\Lambda}}, R_{_{\Gamma}})=0$.
As $\widehat{\mathrm{Ext}}^{-1}_{R_{_{{\Lambda}\cap {\Gamma}}}}(R_{_{\Lambda}}, R_{_{\Gamma}})=0$ and $R_{_{\Lambda}}$ is totally reflexive as an
$R_{_{{\Lambda}\cap {\Gamma}}}$-module one has, by \cite[Lemma 5.8]{am},
the exact sequence
\begin{equation}\label{e48}
0\rightarrow \mathrm{Hom}_{R_{_{{\Lambda}\cap {\Gamma}}}}(R_{_{\Lambda}} , R_{_{{\Lambda}\cap {\Gamma}}})\otimes_{R_{_{{\Lambda}\cap {\Gamma}}}} R_{_{\Gamma}}
 \stackrel{\nu}{\longrightarrow} \mathrm{Hom}_{R_{_{{\Lambda}\cap {\Gamma}}}}(R_{_{\Lambda}}, R_{_{\Gamma}})
  \longrightarrow  \widehat{\mathrm{Ext}}^{0}_{R_{_{{\Lambda}\cap {\Gamma}}}}(R_{_{\Lambda}}, R_{_{\Gamma}})\rightarrow 0\,,
\end{equation}
where the map $\nu$ is given by
$$\nu(\psi\otimes \theta)=\psi_{_\theta}\, , \ \ \  \psi_{_\theta}(\alpha)=\psi(\alpha)\theta\,  .$$
In a similar way to (\ref{e47}), one gets  the natural
 isomorphism
$\tau: \mathrm{Hom}_{R_{_{\Gamma}}}(R_{_{{\Lambda}\cup {\Gamma}}},
R_{_{\Gamma}})\longrightarrow
{\Sigma_{l\in {\Lambda}\cup {\Gamma}} I_l/(\Sigma_{l\in {\Gamma}} I_l)}$
 given by $\tau(\psi)=\psi(\stackrel{\centerdot}{\varphi})$,
where $\stackrel{\centerdot}{\varphi}$
is the identity element of $R_{_{{\Lambda}\cup {\Gamma}}}$.
It is straightforward to show that the following diagram commutes:
$$\begin{array}{ccc}
\mathrm{Hom}_{R_{_{{\Lambda}\cap {\Gamma}}}}(R_{_{\Lambda}} , R_{_{{\Lambda}\cap {\Gamma}}})\otimes_{R_{_{{\Lambda}\cap {\Gamma}}}} R_{_{\Gamma}} &
\stackrel{\nu}{-\hspace{-0.2cm}-\hspace{-0.2cm}-\hspace{-0.2cm}-\hspace{-0.2cm}-\hspace{-0.2cm}
-\hspace{-0.2cm}-\hspace{-0.2cm}-\hspace{-0.2cm}-\hspace{-0.2cm}-\hspace{-0.2cm}-\hspace{-0.2cm}
-\hspace{-0.2cm}-\hspace{-0.2cm}-\hspace{-0.2cm}-\hspace{-0.2cm}-\hspace{-0.2cm}-\hspace{-0.2cm}
-\hspace{-0.2cm}-\hspace{-0.2cm}-\hspace{-0.2cm}-\hspace{-0.2cm}-\hspace{-0.2cm}-\hspace{-0.2cm}
-\hspace{-0.2cm}-\hspace{-0.2cm}-\hspace{-0.2cm}-\hspace{-0.2cm}-\hspace{-0.2cm}-\hspace{-0.2cm}
-\hspace{-0.2cm}-\hspace{-0.2cm}-\hspace{-0.2cm}-\hspace{-0.2cm}-\hspace{-0.2cm}
\longrightarrow} &
\mathrm{Hom}_{R_{_{{\Lambda}\cap {\Gamma}}}}(R_{_{\Lambda}}, R_{_{\Gamma}}) \\
\begin{array}{ll} \kappa\otimes R_{_\Gamma} {\Bigg\downarrow} \cong \end{array} &
        &  \begin{array}{ll} f {\bigg\downarrow} \cong \end{array}\\
{\Sigma_{l\in {\Lambda}} I_l/(\Sigma_{l\in {\Lambda}\cap {\Gamma}} I_l)}\otimes_{R_{_{{\Lambda}\cap {\Gamma}}}} R_{_{\Gamma}} &
 \stackrel{\cong}{\underset{g}{\longrightarrow}}
I_a/(\Sigma_{l\in {\Gamma}} I_aI_l)
\stackrel{\cong}{\underset{h}{\longrightarrow}}
\Sigma_{l\in {\Lambda}\cup {\Gamma}} I_l/(\Sigma_{l\in {\Gamma}} I_l)
\stackrel{\cong}{\underset{\tau}{\longleftarrow}}&
\mathrm{Hom}_{R_{_{\Gamma}}}(R_{_{{\Lambda}\cup {\Gamma}}}, R_{_{\Gamma}})\,,
\end{array}$$
where the maps $f, g$ and $h$ are natural isomorphisms.
Hence  $\nu$ is surjective and  (\ref{e48}) implies that
$\widehat{\mathrm{Ext}}^{0}_{R_{_{{\Lambda}\cap {\Gamma}}}}(R_{_{\Lambda}}, R_{_{\Gamma}})=0$.
\hspace{10.3cm}$\Box$

The following results give a partial converse to Theorem~\ref{T42}.
Note that Proposition~\ref{P43} is a generalization of the result
of Jorgensen et.\! al. \cite[Theorem 3.1]{jls-w}.
\begin{pro}\label{P42}
Let $R$ be a Cohen-Macaulay ring. Assume that there exist a Gorenstein
local ring $Q$  and ideals $I_1, \cdots, I_n$ of $Q$
satisfying the following conditions.
\begin{itemize}
\item[(1)] There is a ring isomorphism $R\cong Q/(I_1+\cdots + I_n)$.
\item[(2)] The ring $R_k=Q/(I_1+\cdots +I_k)$ is Cohen-Macaulay for all $k\in [n]$.
\item[(3)]  $\mathrm{fd}_{R_j}(R_k)<\infty$ for all $k\in [n]$ and all $1\leqslant j \leqslant k$.
\item[(4)]  For each $k\in [n]$ and all $0\leqslant j < k$,
$\mathrm{I}_{R_k}^{R_k}(t)\neq t^e \mathrm{I}_{R_{j}}^{R_{j}}(t)$ for any integer $e$. ($R_0=Q$)
\end{itemize}
Then there exist integers $g_0, g_1, \cdots, g_{n-1}$ such that
$$[\mathrm{Ext}_{Q}^{g_0}(R, Q)] \vartriangleleft [\mathrm{Ext}_{R_1}^{g_1}(R, R_1)]\vartriangleleft \cdots \vartriangleleft
[\mathrm{Ext}_{R_{n-1}}^{g_{n-1}}(R, R_{n-1})]\vartriangleleft [R]$$
is a chain in $\mathfrak{G}_0(R)$ of length $n$.
\end{pro}
\begin{proof}
We prove by induction. For $n=1$, it is clear that  $\mathrm{Ext}_{Q}^{g_0}(R, Q)$ is a dualizing
$R$-module for some integer $g_0$.  It will be shown in following that the condition (4) implies
$[\mathrm{Ext}_{Q}^{g_0}(R, Q)]\vartriangleleft[R]$.
Let $n=2$. As $\mathrm{fd}_{R_1}(R)<\infty$, one has $\mbox{G}$-$\mathrm{dim}_{R_1}(R)<\infty$.
Then, by Remark~\ref{R1}, there exists an integer $g_1$ such that
$\mathrm{Ext}_{R_1}^{i}(R, R_1)=0$ for all $i\neq g_1$ and
$C_1=\mathrm{Ext}_{R_1}^{g_1}(R, R_1)$ is a semidualizing $R$-module.
Therefore there is an isomorphism $C_1\simeq \Sigma^{g_1}\mathbf{R}\mathrm{Hom}_{R_1}(R, R_1)$ in the derived category $\mathrm{D}(R)$.
Thus, by \cite[(1.7.8)]{c2},  $\mathrm{I}_R^{C_1}(t)=t^{-g_1}\mathrm{I}_{R_1}^{R_1}(t)$.
Also there exists an integer $g_0$ such that
$\mathrm{Ext}_{Q}^{i}(R, Q)=0$ for all $i\neq g_0$ and
$D=\mathrm{Ext}_{Q}^{g_0}(R, Q)$ is a dualizing $R$-module and then
$D\simeq \Sigma^{g_0}\mathbf{R}\mathrm{Hom}_Q(R, Q)$ in $\mathrm{D}(R)$.
Assumption (4) implies that $C_1$ is a non-trivial semidualizing $R$-module and so
$[D]\vartriangleleft[C_1]\vartriangleleft[R]$ is a chain in
$\mathfrak{G}_0(R)$ of length $2$.

Let $n>2$ and suppose that the assertion holds true for $n-1$. By induction there exist
integers $h_0, h_1, \cdots, h_{n-2}$ such that
\begin{equation}\label{e49}
[\mathrm{Ext}_Q^{h_0}(R_{n-1}, Q)] \vartriangleleft [\mathrm{Ext}_{R_1}^{h_1}(R_{n-1}, R_1)]\vartriangleleft \cdots \vartriangleleft
[\mathrm{Ext}_{R_{n-2}}^{h_{n-2}}(R_{n-1}, R_{n-2})]\vartriangleleft [R_{n-1}]
\end{equation}
is a chain in $\mathfrak{G}_0(R_{n-1})$ of length $n-1$.
( In fact, there is an isomorphism
$\mathrm{Ext}_{R_i}^{h_i}(R_{n-1}, R_i)\simeq \Sigma^{h_i}\mathbf{R}\mathrm{Hom}_{R_i}(R_{n-1}, R_i)$
in $\mathrm{D}(R_{n-1})$ for all $0\leqslant i\leqslant n-2$.)

As $\mathrm{fd}_{R_k}(R)<\infty$, one has $\mbox{G}$-$\mathrm{dim}_{R_k}(R)<\infty$
for all $k\in [n]$ and so, by Remark~\ref{R1}, there exists an integer $g_k$ such that
$\mathrm{Ext}_{R_k}^{i}(R, R_k)=0$ for all $i\neq g_k$ and
$C_k=\mathrm{Ext}_{R_k}^{g_k}(R, R_k)$ is a semidualizing $R$-module.
We have $C_k\simeq \Sigma^{g_k}\mathbf{R}\mathrm{Hom}_{R_k}(R, R_k)$ in $\mathrm{D}(R)$.
Also there exists an integer $g_0$ such that
$\mathrm{Ext}_{Q}^{i}(R, Q)=0$ for all $i\neq g_0$ and
$D=\mathrm{Ext}_{Q}^{g_0}(R, Q)$  is a dualizing for $R$  and so
$D\simeq \Sigma^{g_0}\mathbf{R}\mathrm{Hom}_Q(R, Q)$ in $\mathrm{D}(R)$.
Note that  there is an isomorphism
$\mathbf{R}\mathrm{Hom}_{R_k}(R, R_k)\simeq \mathbf{R}\mathrm{Hom}_{R_{n-1}}(R, \mathbf{R}\mathrm{Hom}_{R_k}(R_{n-1}, R_k))$, $0\leqslant k \leqslant n-1$,
in $\mathrm{D}(R)$,   and $R$ is a finite $R_{n-1}$-module with
$\mathrm{fd}_{R_{n-1}}(R)<\infty$.
Thus, by \cite[Theorem 5.7]{fs-w1} and (\ref{e49}), one obtains
$[\mathrm{Ext}_{R_{k-1}}^{g_{k-1}}(R, R_{k-1})]\trianglelefteq[\mathrm{Ext}_{R_k}^{g_k}(R, R_k)]$ for all $1\leqslant k \leqslant n-1$.
By \cite[(1.7.8)]{c2},  $\mathrm{I}_R^{C_k}(t)=t^{-g_k}\mathrm{I}_{R_k}^{R_k}(t)$ for all
$1\leqslant k \leqslant n-1$ and $\mathrm{I}_R^{D}(t)=t^{-g_0}\mathrm{I}_{Q}^{Q}(t)$.
Therefore, by condition (4),
$[\mathrm{Ext}_{R_{k-1}}^{g_{k-1}}(R, R_{k-1})]\vartriangleleft[\mathrm{Ext}_{R_k}^{g_k}(R, R_k)]$ for all $1\leqslant k \leqslant n-1$, and
$[\mathrm{Ext}_{R_{n-1}}^{g_{n-1}}(R, R_{n-1})]\vartriangleleft[R]$.
Hence
$$[\mathrm{Ext}_Q^{g_0}(R, Q)] \vartriangleleft [\mathrm{Ext}_{R_1}^{g_1}(R, R_1)]
\vartriangleleft \cdots \vartriangleleft
[\mathrm{Ext}_{R_{n-1}}^{g_{n-1}}(R, R_{n-1})]\vartriangleleft [R]$$
is a chain in $\mathfrak{G}_0(R)$ of length $n$.
\end{proof}

\begin{pro}\label{P43}
Let $R$ be a Cohen-Macaulay ring. Assume that there exist a Gorenstein
local ring $Q$  and ideals $I_1, \cdots, I_n$ of $Q$
satisfying the following conditions.
\begin{itemize}
\item[(1)] There is a ring isomorphism $R\cong Q/(I_1+\cdots + I_n)$.
\item[(2)]  For each $\Lambda\subseteq [n]$,
the ring $R_{_\Lambda}=Q/(\Sigma_{l\in \Lambda} I_l)$ is Cohen-Macaulay.
\item[(3)]  For subsets $\Lambda$, $\Gamma$ of $[n]$ with $\Lambda\cap \Gamma=\emptyset$
\begin{itemize}
\item[(i)] $\mathrm{Tor}^Q_{\geqslant1}(R_{_\Lambda}, R_{_\Gamma})=0;$
\item[(ii)] For all $i \in \mathbb{Z}$, $\hspace{0.3cm}\mathrm{\widehat{Ext}}^i_Q (R_{_\Lambda}, R_{_\Gamma})=0=\mathrm{\widehat{Tor}}^Q_i(R_{_\Lambda}, R_{_\Gamma}).$
\end{itemize}
\item[(4)] For two subsets $\Lambda$, $\Gamma$ of $[n]$ with $\Lambda\neq \Gamma$ and
for any integer $e$,
$\mathrm{I}_{R_{_\Lambda}}^{R_{_\Lambda}}(t)\neq t^e \mathrm{I}_{R_{_\Gamma}}^{R_{_\Gamma}}(t)$.
\end{itemize}
Then, for each $\Lambda\subseteq [n]$, there is an integer $g_{_\Lambda}$ such that
$\mathrm{Ext}^{g_{_\Lambda}}_{R_{_\Lambda}}(R, R_{_\Lambda})$
is a semidualizing $R$-module.
As conclusion, $R$ admits $2^n$ non-isomorphic semidualizing modules.
\end{pro}
\begin{proof}
For two subsets $\Lambda$, $\Gamma$ of $[n]$ with $\Gamma\subseteq \Lambda$, we have
$\mbox{G}$-$\mathrm{dim}_{R_{_\Gamma}}(R_{_\Lambda})<\infty$.
Indeed,  $\mbox{G}$-$\mathrm{dim}_Q(R_{_{\Lambda\setminus\Gamma}})
<\infty$, since $Q$ is Gorenstein. Thus $R_{_{\Lambda\setminus\Gamma}}$
admits a Tate resolution
$\mathbf{T}\stackrel{\vartheta}{\longrightarrow}\mathbf{P}\stackrel{\pi}{\longrightarrow}
R_{_{\Lambda\setminus\Gamma}}$ over $Q$, where
$\vartheta_i$ is isomorphism for all $i\gg 0$.
We show that the induced diagram
$\mathbf{T}\otimes_Q R_{_\Gamma}\stackrel{\vartheta\otimes_Q R_{_\Gamma}}{-\hspace{-0.2cm}-\hspace{-0.2cm}-\hspace{-0.2cm}\longrightarrow}
\mathbf{P}\otimes_Q R_{_\Gamma}
\stackrel{\pi\otimes_Q R_{_\Gamma}}{-\hspace{-0.2cm}-\hspace{-0.2cm}-\hspace{-0.2cm}\longrightarrow}
R_{_{\Lambda\setminus\Gamma}}\otimes_Q R_{_\Gamma}$
is a Tate resolution of
$R_{_{\Lambda\setminus\Gamma}}\otimes_Q R_{_\Gamma}\cong R_{_\Lambda}$
over $R_{_\Gamma}$.
By condition (3)(i),  $\mathbf{P}\otimes_Q R_{_\Gamma}$ is a free resolution of
$R_{_\Lambda}$ over $R_{_\Gamma}$.
Also by assumption,
$\mathrm{\widehat{Tor}}^Q_i(R_{_{\Lambda\setminus\Gamma}}, R_{_\Gamma})=0$ for all
$i \in \mathbb{Z}$ and then $\mathbf{T}\otimes_Q R_{_\Gamma}$ is an exact complex
of finite free $R_{_\Gamma}$-modules.
Of course, the map $\vartheta_i\otimes_Q R_{_\Gamma}$ is an isomorphism for all $i\gg 0$.
In order to show that
$\mathrm{Hom}_{R_{_\Gamma}}(\mathbf{T}\otimes_Q R_{_\Gamma}, R_{_\Gamma})$
is exact we note that the sequence of isomorphisms
$$\mathrm{Hom}_{R_{_\Gamma}}(\mathbf{T}\otimes_Q R_{_\Gamma}, R_{_\Gamma})
\cong \mathrm{Hom}_Q(\mathbf{T}, \mathrm{Hom}_{R_{_\Gamma}}( R_{_\Gamma}, R_{_\Gamma}))\cong  \mathrm{Hom}_Q(\mathbf{T}, R_{_\Gamma}),$$
implies that
$$\mathrm{H}_i(\mathrm{Hom}_{R_{_\Gamma}}(\mathbf{T}\otimes_Q R_{_\Gamma}, R_{_\Gamma}))\cong\mathrm{H}_i(  \mathrm{Hom}_Q(\mathbf{T}, R_{_\Gamma}))
\cong \mathrm{\widehat{Ext}}^{-i}_Q (R_{_{\Lambda\setminus\Gamma}}, R_{_\Gamma})$$
which is zero, by condition (3)(ii), for all  $i \in \mathbb{Z}$.
Hence  the complex
$\mathrm{Hom}_{R_{_\Gamma}}(\mathbf{T}\otimes_Q R_{_\Gamma}, R_{_\Gamma})$
is exact and so
$R_{_\Lambda}$ admits a Tate resolution over $R_{_\Gamma}$. Therefore
$\mbox{G}$-$\mathrm{dim}_{R_{_\Gamma}}(R_{_\Lambda})<\infty$.

In particular, $\mbox{G}$-$\mathrm{dim}_{R_{_\Lambda}}(R)<\infty$ for all $\Lambda\subseteq [n]$.
Hence, by Remark~\ref{R1}, $\mathrm{Ext}^i_{R_{_\Lambda}}(R, R_{_\Lambda})=0$ for all $i\neq g_{_\Lambda}$,
where $g_{_\Lambda}:=\mbox{G}$-$\mathrm{dim}_{R_{_\Lambda}}(R)$, and
$C_{_\Lambda}:=\mathrm{Ext}^{g_{_\Lambda}}_{R_{_\Lambda}}(R, R_{_\Lambda})$
is a semidualizing $R$-module.

Note that there is an isomorphism
$C_{_\Lambda}\simeq \Sigma^{g_{_\Lambda}}\mathbf{R}\mathrm{Hom}_{R_{_\Lambda}}(R, R_{_\Lambda})$ in the derived category $\mathrm{D}(R)$. Therefore, by \cite[(1.7.8)]{c2},
$$\mathrm{I}_R^{C_{_\Lambda}}(t)=\mathrm{I}_R^{\Sigma^{^{g_{_\Lambda}}}\mathbf{R}\mathrm{Hom}_{R_{_\Lambda}}(R, R_{_\Lambda})}(t)=t^{-g_{_\Lambda}}\mathrm{I}_{R_{_\Lambda}}^{R_{_\Lambda}}(t).$$
Now the condition (4) implies that the $2^n$ semidualizing $C_{_\Lambda}$ are
pairwise non-isomorphic.
\end{proof}




\begin{thebibliography}{10}


\bibitem {am} L. L. Avramov and A. Martsinkovsky, {\it Absolute, relative, and Tate cohomology of modules of finite Gorenstein dimension},
Proc. London Math. Soc. {\bf 85} (2002), 393--440.



\bibitem{c2} L. W. Christensen,  {\it Semi-dualizing complexes and their Auslander categories},  Trans. Amer. Math. Soc. {\bf 353} (2001), 1839--1883.



\bibitem{cs-w} L. W. Christensen and S. Sather-Wagstaff, \emph{A Cohen-Macaulay algebra has only finitely many semidualizing modules},
Math. Proc. Camb. Phil. Soc. {\bf 145} (2008), 601--603.


\bibitem{f} H. B. Foxby,  \emph{Gorenstein modules and related modules},  Math. Scand. {\bf 31} (1972),  267--284.


\bibitem{fs-w1} A. Frankild and S. Sather-Wagstaff, {\it Reflexivity and ring homomorphisms of finite flat dimension}, Comm. Algebra {\bf 35} (2007),  461--500.

\bibitem{fs-w2} A. Frankild and S. Sather-Wagstaff, \emph{The set of semidualizing complexes is a nontrivial metric space}, J. Algebra {\bf 308}
(2007), 124--143.


\bibitem{gerko} A. Gerko, {\it On the structure of the set of semidualizing complexes}, Illinois J. Math. {\bf 48}
(2004), 965--976.

\bibitem{g}
E. S. Golod, {\it G-dimension and generalized perfect ideals}, Trudy Mat. Inst. Steklov. {\bf 165} (1984), 62-66.

\bibitem{h} R. Hartshorne, \emph{Residues and duality}, Lecture Notes in Mathematics, No. 20, Springer-Verlag,
Berlin, 1966.


\bibitem{hj}  H. Holm and P. J{\o}rgensen, \emph{Semi-dualizing modules and related Gorenstein homological dimensions},
Journal of Pure and Applied Algebra {\bf 205} (2006), 423--445.


\bibitem{jls-w} D. A. Jorgensen, G. J. Leuschke, and S. Sather-Wagstaff, {\it Presentations of Rings with Non-Trivial Semidualizing Modules},
Collect. Math. {\bf 63} (2012), 165--180.


\bibitem{ns-w}  S. Nasseh and S. Sather-Wagstaff, \emph{A local ring has only finitely many semidualizing complexes up to shift-isomorphism},
 arXiv:1201.0037v2 (2012).

\bibitem{reiten} I. Reiten, \emph{The converse to a theorem of Sharp on Gorenstein modules}, Proc. Amer. Math. Soc. {\bf 32} (1972), 417--420.


\bibitem{s-w1} S. Sather-Wagstaff, {\it Lower bounds for the number of semidualizing complexes
over a local ring},  Math. Scand. {\bf 110} (2012), 5--17.


\bibitem{s-w2} S. Sather-Wagstaff,  {\it Semidualizing modules},  http://www.ndsu.edu/pubweb/\~ssatherw/DOCS/sdm.pdf


\bibitem{s} R. Y. Sharp, \emph{Finitely generated modules of finite iniective dimension over certain Cohen-Macaulay rings}, Proc. London Math. Soc. {\bf 25} (1972), 303--328.

\bibitem{v} W. V. Vasconcelos,  {\it Divisor theory in module categories},  North-Holland Math. Stud., vol. {\bf 14},  North-Holland Publishing
Co.,  Amsterdam,  1974.


\end{thebibliography}
\end{document}